\documentclass[12pt, reqno]{amsart}
\usepackage{amsmath, amsthm, amscd, amsfonts, amssymb, graphicx, color, mathrsfs}
\usepackage[bookmarksnumbered, colorlinks, plainpages]{hyperref}
\usepackage[all]{xy}
\usepackage{slashed}
\usepackage{tipa}
\usepackage{mathabx}
\usepackage{soul}
\usepackage{cancel}
\usepackage{ulem}
\usepackage{inputenc}
\usepackage{booktabs}

\newtheorem{theorem}{Theorem}[section]
\newtheorem{lemma}[theorem]{Lemma}

\newtheorem{proposition}[theorem]{Proposition}
\newtheorem{corollary}[theorem]{Corollary}
\theoremstyle{definition}
\newtheorem{definition}[theorem]{Definition}

\theoremstyle{remark}
\newtheorem{remark}[theorem]{Remark}
\numberwithin{equation}{section}

\begin{document}
\setcounter{page}{1}

\title[ Drift diffusion equations   on compact Lie groups ]{ Drift diffusion equations with fractional diffusion   on compact Lie groups}

\author[D. Cardona]{Duv\'an Cardona}
\address{
  Duv\'an Cardona:
  \endgraf
  Department of Mathematics: Analysis, Logic and Discrete Mathematics
  \endgraf
  Ghent University, Belgium
  \endgraf
  {\it E-mail address} {\rm duvanc306@gmail.com, duvan.cardonasanchez@ugent.be}
  }
  
  \author[J. Delgado]{Julio Delgado}
\address{
  Julio Delgado:
  \endgraf
  Departmento de Matem\'aticas
  \endgraf
  Universidad del Valle
  \endgraf
  Cali-Colombia
    \endgraf
    {\it E-mail address} {\rm delgado.julio@correounivalle.edu.co}
  }

\author[M. Ruzhansky]{Michael Ruzhansky}
\address{
  Michael Ruzhansky:
  \endgraf
  Department of Mathematics: Analysis, Logic and Discrete Mathematics
  \endgraf
  Ghent University, Belgium
  \endgraf
 and
  \endgraf
  School of Mathematical Sciences
  \endgraf
  Queen Mary University of London
  \endgraf
  United Kingdom
  \endgraf
  {\it E-mail address} {\rm michael.ruzhansky@ugent.be, m.ruzhansky@qmul.ac.uk}
  }


\thanks{The authors are supported  by the FWO  Odysseus  1  grant  G.0H94.18N:  Analysis  and  Partial Differential Equations and by the Methusalem programme of the Ghent University Special Research Fund (BOF)
(Grant number 01M01021). J. Delgado is also supported by Vice. Inv. Universidad del Valle Grant CI-71281,  MathAmSud and Minciencias-Colombia under the project MATHAMSUD 21-MATH-03.  Michael Ruzhansky is also supported  by EPSRC grant 
EP/R003025/2.
}

\begin{abstract}
In this work we investigate the well-posedness for  difussion equations associated to subelliptic pseudo-differential operators on compact Lie groups. The diffusion by  strongly elliptic operators is considered as a special case and in particular the fractional diffusion with respect to the Laplacian.  The  general case is studied within the H\"ormander classes associated to a sub-Riemannian structure on the group (encoded by a  H\"ormander system of vector fields). Applications to diffusion equations for fractional sub-Laplacians, fractional powers of more general subelliptic operators, and the corresponding quasi-geostrophic model with drift $D$ are investigated. Examples on $\textnormal{SU}(2)$ for diffusion problems with fractional diffusion are analysed.   
\end{abstract} 
\maketitle
\allowdisplaybreaks
\tableofcontents

\section{Introduction}

In this manuscript we establish the well-posedness of the Cauchy problem 
\begin{equation}\label{PVI:Intro}(\textnormal{IVP}): \begin{cases}\frac{\partial v}{\partial t}-K(t)v=f ,& \text{ }v\in \mathscr{D}'((0,T)\times G),
\\v(0)=u_0 ,& \text{ } \end{cases}
\end{equation} 
on a compact Lie group $G,$ where the the time-dependent family of pseudo-differential operators $K(t)$ includes subelliptic operators, fractional Laplacians and sub-Laplacians with or without drifts (extra perturbations like smooth vector fields).  The motivation for  the consideration of drifts came from quasi-geostrophic theory, see  \cite{caf:annals,knt:geo} and references therein for details.

The investigation of fractional diffusion equations has been intensively active in the last decades,
in particular motivated by the diverse applications e.g fractional dynamics \cite{skb:kk}, \cite{nl:aw},  fluid mechanics \cite{mb:lf}, \cite{pc:lf} and mathematical finance  \cite{c:ff}.  Moreover, the paper  \cite{caf:annals} by Caffarelli and Vasseur has considered the analysis of drift diffusion equations with fractional diffusion inspired by the applications of such models in quasi-geostrophic theory and control theory, see Constantin \cite{pc:lf} and Duvaut and Lions \cite{DuvautLions}, respectively. In the Euclidean setting, analysis of the fractional Laplacian in terms of the pseudo-differential operators theory and singular integrals can be found e.g. in G. Grubb \cite{GGr:aq} (and the references therein). On compact manifolds, in particular, compact Lie groups a theory of initial value problems for pseudo-differential operators can be found in the  book of M. Taylor \cite{Taylorbook1981} and in Chazarain and  Piriou \cite{chpi:book}. 
 
On the other hand, the research on the fractional Laplacian and the fractional diffusion on different group structures have been of recent interest, for instance  the discrete fractional diffusion is considered in  \cite{ron:ddf}. The fractional Laplacian on the torus has been studied in  \cite{ron:t} and several aspects of the harmonic analysis of the fractional Laplacian on lattices has been investigated in \cite{ron:di}. Other recent works on the fractional Laplacian and its different generalisations have been considered  in \cite{Uhlm:fl},   \cite{dab:geo}, \cite{del:deghyp1} \cite{del:depar1}, \cite{val:p1}, \cite{val:p2}, \cite{eps:aw}, \cite{rup:f1}, \cite{GGr:aq}, \cite{ppop:fl},  \cite{val:p3} and the references therein. In particular, the counterpart of the results proved here for compact Lie groups have been analysed on graded Lie groups in \cite{diffgLie:p1} in the case of hypoelliptic problems for Rockland operators.  A nice exposition on the fractional Sobolev spaces and the  fractional Laplacian can be found in \cite{val:p1}. On the other hand, drift diffusion equations with fractional diffusion  have intensely attracted the interest in the last 10 years starting with the works of  Caffarelli and Vasseur in \cite{caf:annals}.

Now, we are going to explain in more detail, the family of initial value problems \eqref{PVI:Intro} covered by our approach. For a compact Lie group  $G,$ the set  $\widehat{G}$ denotes its unitary dual. Motivated by the applications of the previous paragraph and by the recent techniques in the analysis of differential problems on compact Lie groups \cite{CardonaRuzhanskyC,Ruz}, built on the harmonic analysis of operators on the phase space $G\times\widehat{G},$ in this paper we analyse the existence and uniqueness for the Cauchy problem
\eqref{PVI:Intro}   associated to a  time-dependent  continuous family of pseudo-differential operators $K(t)$ in the H\"ormander classes $\Psi^{m,\mathcal{L}}_{\rho,\delta}(G\times \widehat{G})$ (introduced in \cite{CardonaRuzhanskyC} and that include e.g. the H\"ormander sub-Laplacian $\mathcal{L}=-\sum_{j=1}^kX_j^2$ and its functional calculus, its fractional powers, etc.) with positive order $m>0$. The functional calculus of the pseudo-differential operators in  \cite{CardonaRuzhanskyC,RuzhanskyWirth2014} allows us to consider the case of the fractional diffusion and of the fractional diffusion with drift. Our analysis includes the case where
\begin{itemize}
    \item $K(t)$ is strongly elliptic with respect to $\mathcal{L},$ of  arbitrary order $m>0.$ Examples are fractional operators of the form $K(t)=a(x,t)\mathcal{L}^{\frac{m}{2}},$ $a(x,t)(1+\mathcal{L})^{\frac{m}{2}},$ where $a\in C([0,T], C^{\infty}(G))$ is a real-valued  function such that $a(x,t)\neq 0,$ for all $(x,t)\in G\times [0,T].$ Our analysis also covers the case where $K(t)=a(x,t)\mathcal{L}^{\frac{m}{2}}_G$ with $\mathcal{L}_G=-\sum_{j=1}^{k}X_j^2$  being the Laplace-Beltrami operator on $G.$ 
    \item $K(t)=a(x,t)\mathcal{L}_{G}^{\frac{m}{2}}+D,$ where $0\leq m\leq 1,$ and the drift $D=\sum_{i=1}^{n}a_{i}X_i\in \textnormal{Diff}^{1}(G)$  is a left-invariant differential operator of first order such that the matrix-valued  symbol of $$-K(t)=-(a(x,t)\mathcal{L}_{G}^{\frac{m}{2}}+D)$$ is positive on every representation space. Moreover, the positivity condition is removed when $m>1,$ and $a(x,t)\neq 0$ for all $(x,t)\in G\times [0,T].$ In both  cases we say that the  diffusion problem \eqref{PVI} has drift $D$ following the terminology e.g. in \cite{caf:annals}.
    \end{itemize}
In many of the cases, the study of the Cauchy problem
\begin{equation}\label{Cauchy:on:M}
    Lu=f,\,\partial_{t}^{j}u(x,0)=g_{j+1},\,0\leq j\leq m-1,
\end{equation}where $ L=\partial_t^m-\sum_{j=0}^{m-1}a_{m-j}(x,t,D)\partial_{t}^j,$ 
with any $a_{m-j}(x,t,D)$ being a differential operator of order $m-j$ on a closed  manifold $M$ (or on $\mathbb{R}^n$) can be reduced to an equivalent first order system of the form
\begin{equation}\label{simplified} \begin{cases}\frac{\partial u}{\partial t}-K(t)u=f ,& \text{ }u\in \mathscr{D}'((0,T)\times M),
\\u(0)=u_0 .& \text{ } \end{cases}
\end{equation}Indeed, if $\Delta_M$ denotes the positive Laplacian on $M,$ and $\Gamma=(1+\Delta_M)^{\frac{1}{2}}$ is the corresponding Bessel potential, setting
\begin{equation}
    u_j:=\partial_t^{j-1}\Gamma^{m-j}u,\,\,1\leq j\leq m,
\end{equation}then \eqref{Cauchy:on:M} is equivalent to the system
\begin{equation}\label{matrix:Rep}
 \partial_t
 \begin{pmatrix} 
 u_1&\\
 u_2&\\
 \vdots&\\
 u_{m-1}&\\
 u_{m}
 \end{pmatrix}= \begin{pmatrix}
0 & \Gamma & 0 &\cdots & 0\\
0 & 0 & \Gamma &\cdots & 0\\
\vdots  &\vdots &\vdots & \ddots & \vdots\\
0  & 0 & 0 & \ddots & \Gamma\\
b_{1}(x,t,D)  & b_{2}(x,t,D) & \cdots & \cdots & b_{m}(x,t,D)\\
\end{pmatrix} \begin{pmatrix} 0&\\
0&\\
 \vdots&\\0&\\ \,f\end{pmatrix},  
\end{equation}where the sequence $b_{m}(x,t,D)=a_{m-j+1}(x,t,D)\Gamma^{j-m}$ of pseudo-differential operators  are of order one, with the functions $u_j$ having initial conditions $u_j|_{t=0}=\Gamma^{m-j}g_j$ at $t=0.$ Changing the notation and denoting the function vector $$(u_1,\cdots ,u_m)^{t}$$ by $u$ and the matrix in \eqref{matrix:Rep} by $K(t)$ we have the equivalence of \eqref{matrix:Rep}, \eqref{simplified} and the initial higher order initial value problem \eqref{Cauchy:on:M}. For a complete analysis of this equivalence  in the setting of parabolic and hyperbolic problems we refer the reader to the book of M. Taylor \cite[Chapters IV-V]{Taylorbook1981}.

The matrix $K(t)$ in \eqref{simplified} has as entries pseudo-differential operators or order one and it illustrates the relevance of the analysis of the initial value problem \eqref{PVI:Intro}, even in the case of a compact Lie group $M=G$. Our choice of the setting of compact Lie groups is motivated by exploiting the group Fourier transform, and the global notion of a matrix-valued symbol that allows the consideration of a variety of pseudo-differential operators $K(t),$ \cite{Ruz}, allowing a coordinates-free approach. Of particular interest for this work, is the case when $K(t)$ becomes (a parametrised in $t\in [0,T]$) family of subelliptic pseudo-differential operators \cite{CardonaRuzhanskyC}. 

The main results of this work are Theorem \ref{Main:Th} in Subsection \ref{subsectio:case:1} and Theorem \ref{Main:Th:2} in Subsection \ref{subsectio:case:1} where we prove that
\begin{itemize}
    \item The Cauchy problem \ref{PVI:Intro} is well posedness for any continuous family of strongly elliptic pseudo-differential operators $K(t)$ in the subelliptic H\"ormander classes  in the H\"ormander classes $\Psi^{m,\mathcal{L}}_{\rho,\delta}(G\times \widehat{G})$ with the solution operator
    \begin{equation}\label{Solution:mapping}
        (v_0,f)\in H^{s,\mathcal{L}}(G)\times L^2([0,T],H^{s,\mathcal{L}}(G))\mapsto u\in L^\infty([0,T],H^{s,\mathcal{L}}(G)),
    \end{equation}being a continuous mapping for any $s\in \mathbb{R}.$ The continuity of the operator is proved with the corresponding energy estimates and the Sobolev spaces $H^{s,\mathcal{L}}(G)$ are then adapted to the regularity scale of the sub-Laplacian $\Vert w\Vert_{H^{s,\mathcal{L}}(G)}:=\Vert (1+\mathcal{L})^{\frac{s}{2}}\Vert_{L^2(G)}.$
    \item The strongly elliptic condition in the continuity of \eqref{Solution:mapping} can be relaxed assuming the positivity of the matrix-valued symbol of the operator $-K(t),$ when the order $m$ is assumed to be in the interval $[0,\rho/\kappa-(2-1/\kappa)\delta].$ The upper bound of this interval, or the maximum order of tolerance for the order, is related with the sharp-G\r{a}rding inequality for the classes $\Psi^{m,\mathcal{L}}_{\rho,\delta}(G\times \widehat{G})$ (see \cite{SharpGarding}). 
    
    In the case where the Laplacian takes the place of the sub-Laplacian, we recover the H\"ormander classes $ \Psi^{m}_{\rho,\delta}(G\times \widehat{G})$ in \cite{Ruz} (and then, the H\"ormander classes defined by localisations \cite{Hormander1985III})  we  allow the order
    $  0\leq m\leq \rho-\delta ,$ and even the criterion in Theorem  \ref{Main:Th:2} for a type of  operators of the form $K(t)=a(t,x)(\mathcal{L}_G)^{\frac{1}{2}}+D,$ $D\in TG^{\mathbb{C}}=TG\otimes i TG,$ (see Corollary \ref{Corollary:general:drift} for details) provides new results beyond  the known criteria given e.g. in  Taylor \cite[Chapters IV-V]{Taylorbook1981}. 
    In that case, the well posedness for the (quasi-geostrophic) differential problem
    $$  v_t= a(t,x)(\mathcal{L}_G)^{\frac{1}{2}}v+\sum_{j=1}^{n}a_{j}(t,x)iX_{j}u,\,v_{0}=u, $$
    is analysed, taking care of the drift $D=\sum_{j=1}^{n}a_{j}(t,x)iX_{j}\in TG^{\mathbb{C}}=TG\otimes i TG.$

    A concrete example in this direction is given in Subsection \ref{Ex:SU2section} in the case of $\textnormal{SU}(2),$ showing  the viability of the positivity criterion of the symbol of the operator $-K(t)$ in   Theorem \ref{Main:Th:2}.
    \item Our criteria for the strongly elliptic case where $K(t)\in \Psi^{m}_{\rho,\delta}(G\times \widehat{G})$ recovers the well posedness result for hyperbolic problems in Taylor \cite{Taylorbook1981} in the case of the fractional Laplacian on $G$.
\end{itemize}
This paper is organised as follows. 
\begin{itemize}
    \item In Section  \ref{Sect2} we provide the preliminaries about the Fourier analysis on compact Lie groups  and the related topics about the subelliptic H\"ormander classes $\Psi^{m}_{\rho,\delta}(G\times \widehat{G}).$
    \item Finally, Section  \ref{GST} is dedicated to the proof of our main results and to presenting several consequences, finalising the section with the investigation of fractional diffusion models (with drift) on $G=\textnormal{SU}(2).$ 
\end{itemize}

\section{Fourier analysis on Lie groups and subelliptic H\"ormander classes}\label{Sect2}

\subsection{Pseudo-differential operators via localisations}\label{S2.1}
 Pseudo-differential operators on compact manifolds, and consequently on compact Lie groups, can be defined by using local coordinate charts (see H\"ormander \cite{Hormander1985III} and also M. Taylor \cite{Taylorbook1981} for a good introductory background on the subject). Although here we will use the matrix-symbols introduced in \cite{Ruz}, as we will see later there is a global description of the H\"ormander classes in \cite{Hormander1985III} in terms of such a global notion of matrix-symbols.
 
Let us briefly  introduce these  classes starting with the definition in the Euclidean setting. Let $U$ be an open  subset of $\mathbb{R}^n.$ We  say that  the ``symbol" $a\in C^{\infty}(U\times \mathbb{R}^n, \mathbb{C})$ belongs to the H\"ormander class of order $m$ and of $(\rho,\delta)$-type,  $S^m_{\rho,\delta}(U\times \mathbb{R}^n),$ $0\leqslant \rho,\delta\leqslant 1,$ if for every compact subset $K\subset U$ and for all $\alpha,\beta\in \mathbb{N}_0^n$, the symbol inequalities
\begin{equation*}
  |\partial_{x}^\beta\partial_{\xi}^\alpha a(x,\xi)|\leqslant C_{\alpha,\beta,K}(1+|\xi|)^{m-\rho|\alpha|+\delta|\beta|},
\end{equation*} hold true uniformly in $x\in K$ for all  $\xi\in \mathbb{R}^n.$ Then, a continuous linear operator $A:C^\infty_0(U) \rightarrow C^\infty(U)$ 
is a pseudo-differential operator of order $m$ of  $(\rho,\delta)$-type, if there exists
a symbol $a\in S^m_{\rho,\delta}(U\times \mathbb{R}^n)$ such that
\begin{equation*}
    Af(x)=\smallint\limits_{\mathbb{R}^n}e^{2\pi i x\cdot \xi}a(x,\xi)(\mathscr{F}_{\mathbb{R}^n}{f})(\xi)d\xi,
\end{equation*} for all $f\in C^\infty_0(U),$ where
$$
    (\mathscr{F}_{\mathbb{R}^n}{f})(\xi):=\smallint\limits_Ue^{-i2\pi x\cdot \xi}f(x)dx
$$ is the  Euclidean Fourier transform of $f$ at $\xi\in \mathbb{R}^n.$ 

Once the definition of H\"ormander classes on open subsets of $\mathbb{R}^n$ is established, it can be extended to smooth manifolds as follows.  Given a $C^\infty$-manifold $M,$ a linear continuous operator $A:C^\infty_0(M)\rightarrow C^\infty(M) $ is a pseudo-differential operator of order $m$ of $(\rho,\delta)$-type, with $ \rho\geqslant   1-\delta, $ and $0\leq \delta<\rho\leq 1,$  if for every local  coordinate patch $\omega: M_{\omega}\subset M\rightarrow U_{\omega}\subset \mathbb{R}^n,$
and for every $\phi,\psi\in C^\infty_0(U_\omega),$ the operator
\begin{equation*}
    Tu:=\psi(\omega^{-1})^*A\omega^{*}(\phi u),\,\,u\in C^\infty(U_\omega),\footnote{As usually, $\omega^{*}$ and $(\omega^{-1})^*$ are the pullbacks, induced by the maps $\omega$ and $\omega^{-1},$ respectively.}
\end{equation*} is a standard pseudo-differential operator with symbol $a_T\in S^m_{\rho,\delta}(U_\omega\times \mathbb{R}^n).$ In this case we write $A\in \Psi^m_{\rho,\delta}(M,\textnormal{loc}).$

\subsection{The positive sub-Laplacian and pseudo-differential operators via global symbols} Let $G$ be a compact Lie group with Lie algebra $\mathfrak{g}\simeq T_{e_G}G$, where $e_{G}$ is the neutral element of $G$, and let  
$$X=\{X_1,\cdots,X_{k} \}\subset \mathfrak{g}$$ 
be a system of $C^\infty$-vector fields. For all $I=(i_1,\cdots,i_\omega)\in \{1,2,\cdots,k\}^{\omega}$ of length $\omega\geqslant   1$, we denote by $$X_{I}:=[X_{i_1},[X_{i_2},\cdots [X_{i_{\omega-1}},X_{i_\omega}]\cdots]]$$
a commutator of length $\omega$, where $X_{I}:=X_{i}$ when $\omega=1$ and $I=(i)$. The system $X$ is said to satisfy H\"ormander's condition of step (or order) $\kappa$ if $\mathfrak{g}=\mathrm{span}\{X_I: |I|\leq \kappa\}$, that is, in other words, the vector fields $X_j$, $j=1,\ldots, k$, together with their commutator up to length $\kappa$, generate the whole Lie algebra $\mathfrak{g}$. 

Note that we are assuming that there is no subsystem $Y=\{Y_1,\cdots,Y_{\ell} \}\subset X$, $\ell< k$, of smooth vector fields such that $\mathfrak{g}=\mathrm{span}\{Y_I: |I|\leq \kappa\}$.
In this case we call $X$ a system of H\"ormander's vector fields.

Given a system $X=\{X_1,\cdots,X_{k}\}$ of H\"ormander's vector fields, then the operator defined as
\begin{equation*}
    \mathcal{L}\equiv \mathcal{L}_{X}:=-(X_{1}^2+\cdots +X_{k}^2),
\end{equation*} 
is a hypoelliptic operator by H\"ormander theorem on sums of the  squares of vector fields (see H\"ormander \cite{Hormander1967}). In particular the operator $\mathcal{L}$ is also subelliptic, and it is called the subelliptic Laplacian associated with the system $X$, or simply sub-Laplacian. It is clear from the definition that one can define different sub-Laplacians by using  different systems of  H\"ormander's vector fields (and that satisfy H\"ormander condition of different step). 

As already mentioned in the introduction, we will make use of the quantization process developed by the third author and V. Turunen in \cite{Ruz}. We briefly recall below how this global quantization is defined. 

Let  $A$ be a continuous linear operator from $C^\infty(G)$ into $\mathscr{D}'(G),$ and let  $\widehat{G}$ be  the algebraic unitary dual of $G.$ Then, there exists a function \begin{equation}\label{symbol}a:G\times \widehat{G}\rightarrow \cup_{\ell\in \mathbb{N}} \mathbb{C}^{\ell\times \ell},\end{equation}  that we call the symbol of $A,$ such that $a(x,\xi):=a(x,[\xi])\in \mathbb{C}^{d_\xi\times d_\xi}$ for every equivalence class $[\xi]\in \widehat{G},$ where $\xi:G\rightarrow \textnormal{Hom}(H_{\xi}),$ $H_{\xi}\cong \mathbb{C}^{d_\xi},$ and such that
\begin{equation}\label{RuzhanskyTurunenQuanti}
    Af(x)=\sum_{[\xi]\in \widehat{G}}d_\xi{\textnormal{Tr}}[\xi(x)a(x,\xi)\widehat{f}(\xi)],\,\,\forall f\in C^\infty(G).
\end{equation}Note that we have denoted by
\begin{equation*}
    \widehat{f}(\xi)\equiv (\mathscr{F}f)(\xi):=\smallint\limits_{G}f(x)\xi(x)^*dx\in  \mathbb{C}^{d_\xi\times d_\xi},\,\,\,[\xi]\in \widehat{G},
\end{equation*} the Fourier transform of $f$ at $\xi\cong(\xi_{ij})_{i,j=1}^{d_\xi},$ where the matrix representation of $\xi$ is induced by an orthonormal basis of the representation space $H_{\xi}.$
The function $a$ in \eqref{symbol} satisfying \eqref{RuzhanskyTurunenQuanti} is unique, and satisfies the identity
\begin{equation*}
    a(x,\xi)=\xi(x)^{*}(A\xi)(x),\,\, A\xi:=(A\xi_{ij})_{i,j=1}^{d_\xi},\,\,\,[\xi]\in \widehat{G}.
\end{equation*}
Note that the previous identity is well defined. Indeed, it is well known that the functions $\xi_{ij}$, which are of $C^\infty$-class, are the eigenfunctions of the positive Laplace operator $\mathcal{L}_G$, that is $\mathcal{L}_G\xi_{ij}=\lambda_{[\xi]}\xi_{ij}$ for some non-negative real number $\lambda_{[\xi]}\geq 0$
depending  only of the equivalence class $[\xi]$ and not on the representation $\xi$.

In general, we refer to the function $a$ as the (global or full) {\it{symbol}} of the operator $A,$ and we will use the notation $A=\textnormal{Op}(a)$ to indicate that $a:=\sigma_A$ is the symbol associated with the operator $A.$

In order to classify symbols in the H\"ormander classes, in  \cite{Ruz} the authors defined the notion of {\it{difference operators}}, which endows $\widehat{G}$ with a difference structure.  Following  \cite{RuzhanskyWirth2015}, a difference operator $Q_\xi: \mathscr{D}'(\widehat{G})\rightarrow \mathscr{D}'(\widehat{G})$ of order $k$  is defined as
\begin{equation}\label{taylordifferences}
    Q_\xi\widehat{f}(\xi)=\widehat{qf}(\xi),\,[\xi]\in \widehat{G},
\end{equation}  for some function $q$ vanishing of order $k$ at the neutral element $e=e_G.$ We will denote by $\textnormal{diff}^k(\widehat{G})$  the class of all difference operators of order $k.$ For a  fixed smooth function $q,$ the associated difference operator will be denoted by $\Delta_q\equiv Q_\xi.$ A system  of difference operators (see e.g. \cite{RuzhanskyWirth2015})
\begin{equation*}
  \Delta_{\xi}^\alpha:=\Delta_{q_{(1)}}^{\alpha_1}\cdots   \Delta_{q_{(i)}}^{\alpha_{i}},\,\,\alpha=(\alpha_j)_{1\leqslant j\leqslant i}, 
\end{equation*}
with $i\geq n$, is called   admissible  if
\begin{equation*}
    \textnormal{rank}\{\nabla q_{(j)}(e):1\leqslant j\leqslant i \}=\textnormal{dim}(G), \textnormal{   and   }\Delta_{q_{(j)}}\in \textnormal{diff}^{1}(\widehat{G}).
\end{equation*}
An admissible collection is said to be strongly admissible if, additionally, 
\begin{equation*}
    \bigcap_{j=1}^{i}\{x\in G: q_{(j)}(x)=0\}=\{e_G\}.
\end{equation*}

\begin{remark}\label{remarkD}
Matrix components of unitary representations induce difference operators. Indeed, if $\xi_{1},\xi_2,\cdots, \xi_{k},$ are  fixed irreducible and unitary  representation of $G$, which not necessarily belong to the same equivalence class, then each coefficient of the matrix
\begin{equation}
 \xi_{\ell}(g)-I_{d_{\xi_{\ell}}}=[\xi_{\ell}(g)_{ij}-\delta_{ij}]_{i,j=1}^{d_{\xi_\ell}},\, \quad g\in G, \,\,1\leq \ell\leq k,
\end{equation} 
that is each function 
$q^{\ell}_{ij}(g):=\xi_{\ell}(g)_{ij}-\delta_{ij}$, $ g\in G,$ defines a difference operator
\begin{equation}\label{Difference:op:rep}
    \mathbb{D}_{\xi_\ell,i,j}:=\mathscr{F}(\xi_{\ell}(g)_{ij}-\delta_{ij})\mathscr{F}^{-1}.
\end{equation}
We can fix $k\geq \mathrm{dim}(G)$ of these representations in such a way that the corresponding  family of difference operators is admissible, that is, 
\begin{equation*}
    \textnormal{rank}\{\nabla q^{\ell}_{i,j}(e):1\leqslant \ell\leqslant k \}=\textnormal{dim}(G).
\end{equation*}
To define higher order difference operators of this kind, let us fix a unitary irreducible representation $\xi_\ell$.
Since the representation is fixed we omit the index $\ell$ of the representations $\xi_\ell$ in the notation that will follow.
Then, for any given multi-index $\alpha\in \mathbb{N}_0^{d_{\xi_\ell}^2}$, with 
$|\alpha|=\sum_{i,j=1}^{d_{\xi_\ell}}\alpha_{i,j}$, we write
$$\mathbb{D}^{\alpha}:=\mathbb{D}_{1,1}^{\alpha_{11}}\cdots \mathbb{D}^{\alpha_{d_{\xi_\ell},d_{\xi_\ell}}}_{d_{\xi_\ell}d_{\xi_\ell}}
$$ 
for a difference operator of order $|\alpha|$.
\end{remark}
The difference operators endow the unitary dual $\widehat{G}$ with a difference structure. For difference operators of the previous form, the  following finite Leibniz-like formula holds true (see  \cite{RuzhanskyTurunenWirth2014} for details). Note that below we are still assuming that the representation $\xi_\ell$ is fixed.
\begin{proposition}[Leibniz rule for difference operators]\label{Leibnizrule} Let $G$ be a compact Lie group and let $\mathbb{D}^{\alpha},$ $\alpha\in \mathbb{N}^{d_{\xi_\ell}}_0,$ be the family of difference operators defined in  \eqref{Difference:op:rep}. Then, the following Leibniz rule
\begin{align*}
   ( \mathbb{D}^{\alpha}(a_{1}a_{2}))(x_0,\xi)=\sum_{ |\gamma|,|\varepsilon|\leqslant |\alpha|\leqslant |\gamma|+|\varepsilon| }C_{\varepsilon,\gamma}(\mathbb{D}^{\gamma}a_{1})(x_0,\xi) (\mathbb{D}^{\varepsilon}a_{2})(x_0,\xi), \quad x_{0}\in G,
\end{align*}holds  for all $a_{1},a_{2}\in C^{\infty}(G)\times \mathscr{S}'(\widehat{G})$, where the summation is taken over all $\varepsilon, \gamma$ such that $|\varepsilon|,|\delta|\leq |\alpha|\leq |\gamma|+|\varepsilon|$. 
\end{proposition}
Note that for a different kind of difference operators, namely for those given by compositions of difference operators of higher order associated with different representations, a Leibniz-like formula still holds true by iteration. For more details about difference operators and Leibniz-like formulas for admissible collections see also Corollary 5.13 in \cite{Fischer2015}.
\medskip

We are now going to introduce the global H\"ormander classes of symbols defined in \cite{Ruz}.
First let us recall that every left-invariant vector field  $Y\in\mathfrak{g}$ can be identified with the first order  differential operator $\partial_{Y}:C^\infty(G)\rightarrow \mathscr{D}'(G)$  given by
 \begin{equation*}
   \partial_{Y}f(x)=  (Y_{x}f)(x)=\frac{d}{dt}f(x\exp(tY) )|_{t=0}.
 \end{equation*}If $\{X_1,\cdots, X_n\}$ is a basis of the Lie algebra $\mathfrak{g},$ then we will use the standard multi-index notation
 $$ \partial_{X}^{\alpha}=X_{x}^{\alpha}=\partial_{X_1}^{\alpha_1}\cdots \partial_{X_n}^{\alpha_n},     $$
 for a canonical left-invariant differential operator of order $|\alpha|.$

By using this property, together with the following notation for the so-called  elliptic weight $$\langle\xi \rangle:=(1+\lambda_{[\xi]})^{1/2},\,\,[\xi]\in \widehat{G},$$ we can finally give the definition of global symbol classes.
\begin{definition}Let $G$ be a compact Lie group and let $0\leqslant \delta,\rho\leqslant 1.$ Let $$\sigma:G\times \widehat{G}\rightarrow \bigcup_{[\xi]\in \widehat{G}}\mathbb{C}^{d_\xi\times d_\xi},$$ be a matrix-valued function such that for any $[\xi]\in \widehat{G},$ $\sigma(\cdot,[\xi])$ is of $C^\infty$-class, and such that, for any given $x\in G$ there is a distribution $k_{x}\in \mathscr{D}'(G),$ smooth in $x,$ satisfying $\sigma(x,\xi)=\widehat{k}_{x}(\xi),$ $[\xi]\in \widehat{G}$. We say that $\sigma \in \mathscr{S}^{m}_{\rho,\delta}(G)$ if the following symbol inequalities 
\begin{equation}\label{HormanderSymbolMatrix}
   \Vert \partial_{X}^\beta \Delta_\xi^{\gamma} \sigma_A(x,\xi)\Vert_{\textnormal{op}}\leqslant C_{\alpha,\beta}
    \langle \xi \rangle^{m-\rho|\gamma|+\delta|\beta|},
\end{equation} are satisfied for all $\beta$ and  $\gamma $ multi-indices and for all $(x,[\xi])\in G\times \widehat{G}.$ For $\sigma_A\in \mathscr{S}^m_{\rho,\delta}(G)$ we will write $A\in\Psi^m_{\rho,\delta}(G)\equiv\textnormal{Op}(\mathscr{S}^m_{\rho,\delta}(G)).$
\end{definition}
The global H\"ormander classes on compact Lie groups can be used to describe the H\"ormander classes defined by local coordinate systems. We present the corresponding statement as follows. 
\begin{theorem}[Equivalence of classes, \cite{Ruz,RuzhanskyTurunenWirth2014}] Let $A:C^{\infty}(G)\rightarrow\mathscr{D}'(G)$ be a continuous linear operator and let $0\leq \delta<\rho\leq 1,$ with $\rho\geq 1-\delta.$ Then, $A\in \Psi^m_{\rho,\delta}(G,\textnormal{loc}),$ if and only if $\sigma_A\in \mathscr{S}^m_{\rho,\delta}(G).$ Consequently,
\begin{equation}\label{EQequivalence}
   \textnormal{Op}(\mathscr{S}^m_{\rho,\delta}(G))= \Psi^m_{\rho,\delta}(G,\textnormal{loc}),\,\,\,0\leqslant \delta<\rho\leqslant 1,\,\rho\geqslant   1-\delta.
\end{equation}
\end{theorem}

\subsection{Subelliptic H\"ormander classes on compact Lie groups} 
 In order to define the subelliptic H\"ormander calculus, we will use a suitable basis of the Lie algebra arising from Taylor expansions.  We explain the choice of this basis by means of the following lemma (see Lemma 7.4 in \cite{Fischer2015}).

 \begin{lemma}\label{Taylorseries} Let $G$ be a compact Lie group of dimension $n.$ Let $\mathfrak{D}=\{\Delta_{q_{(j)}}\}_{1\leqslant j\leqslant n}$ be a strongly admissible admissible collection of difference operators, that is
\begin{equation*}
    \textnormal{rank}\{\nabla q_{(j)}(e):1\leqslant j\leqslant n \}=n, \,\,\,\bigcap_{j=1}^{n}\{x\in G: q_{(j)}(x)=0\}=\{e_G\}.
\end{equation*}
Then there exists a basis $X_{\mathfrak{D}}=\{X_{1,\mathfrak{D}},\cdots ,X_{n,\mathfrak{D}}\}$ of $\mathfrak{g}$ such that $$X_{j,\mathfrak{D}}q_{(k)}(\cdot^{-1})(e_G)=\delta_{jk}.
$$
Moreover, by using the multi-index notation $$\partial_{X}^{(\beta)}=\partial_{X_{1,\mathfrak{D}}}^{\beta_1}\cdots \partial_{X_{n,\mathfrak{D}}}^{\beta_n}, $$ for any $\beta\in\mathbb{N}_0^n,$
where $$\partial_{X_{i,\mathfrak{D}}}f(x)=  \frac{d}{dt}f(x\exp(tX_{i,\mathfrak{D}}) )|_{t=0},\,\,f\in C^{\infty}(G),$$ and denoting by
\begin{equation*}
    R_{x,N}^{f}(y)=f(xy)-\sum_{|\alpha|<N}q_{(1)}^{\alpha_1}(y^{-1})\cdots q_{(n)}^{\alpha_n}(y^{-1})\partial_{X}^{(\alpha)}f(x)
\end{equation*} 
the Taylor remainder, we have that 
\begin{equation*}
    | R_{x,N}^{f}(y)|\leqslant C|y|^{N}\max_{|\alpha|\leqslant N}\Vert \partial_{X}^{(\alpha)}f\Vert_{L^\infty(G)},
\end{equation*}
where the constant $C>0$ is dependent on $N,$ $G$ and $\mathfrak{D}$ (but not on $f\in C^\infty(G)).$ In addition we have that $\partial_{X}^{(\beta)}|_{x_1=x}R_{x_1,N}^{f}=R_{x,N}^{\partial_{X}^{(\beta)}f}$, and 
\begin{equation*}
    | \partial_{X}^{(\beta)}|_{y_1=y}R_{x,N}^{f}(y_1)|\leqslant C|y|^{N-|\beta|}\max_{|\alpha|\leqslant N-|\beta|}\Vert \partial_{X}^{(\alpha+\beta)}f\Vert_{L^\infty(G)},
\end{equation*}provided that $|\beta|\leqslant N.$
 \end{lemma}

Using the notation above, and denoting by $\Delta_{\xi}^\alpha:=\Delta_{q_{(1)}}^{\alpha_1}\cdots   \Delta_{q_{(n)}}^{\alpha_{n}},$ we can introduce the subelliptic H\"ormander class of symbols of order $m\in \mathbb{R}$ of type $(\rho,\delta)$.

\begin{definition}[Subelliptic H\"ormander classes]\label{contracted''}
   Let $G$ be a compact Lie group and let $0\leqslant \delta,\rho\leqslant 1.$ Let us consider a sub-Laplacian $\mathcal{L}=-(X_1^2+\cdots +X_{k}^2)$ on $G,$ where the system of vector fields $X=\{X_i\}_{i=1}^{k}$ satisfies the H\"ormander condition of step $\kappa$.  We say that $\sigma \in {S}^{m,\,\mathcal{L}}_{\rho,\delta}(G\times \widehat{G})$ if 
   \begin{equation}\label{InIC}
      p_{\alpha,\beta,\rho,\delta,m,\textnormal{left}}(\sigma)':= \sup_{(x,[\xi])\in G\times \widehat{G} }\Vert \widehat{ \mathcal{M}}(\xi)^{(\rho|\alpha|-\delta|\beta|-m)}\partial_{X}^{(\beta)} \Delta_{\xi}^{\alpha}\sigma(x,\xi)\Vert_{\textnormal{op}} <\infty,
   \end{equation}
   \begin{equation}\label{InIIC}
      p_{\alpha,\beta,\rho,\delta,m,\textnormal{right}}(\sigma)':= \sup_{(x,[\xi])\in G\times \widehat{G} }\Vert (\partial_{X}^{(\beta)} \Delta_{\xi}^{\alpha} \sigma(x,\xi) ) \widehat{ \mathcal{M}}(\xi)^{(\rho|\alpha|-\delta|\beta|-m)}\Vert_{\textnormal{op}} <\infty,
   \end{equation} holds true for all $\alpha, \beta\in \mathbb{N}^n_0$.
  \end{definition}
By following the usual nomenclature, we  define:
\begin{equation*}
    \textnormal{Op}({S}^{m,\,\mathcal{L}}_{\rho,\delta}(G\times \widehat{G})):=\{A:C^{\infty}(G)\rightarrow \mathscr{D}'(G):\sigma_A\equiv\widehat{A}(x,\xi)\in {S}^{m,\,\mathcal{L}}_{\rho,\delta}(G\times \widehat{G}) \},
\end{equation*} with
\begin{equation*}
    Af(x)=\sum_{[\xi]\in \widehat{G}}d_\xi \textnormal{{Tr}}(\xi(x)\widehat{A}(x,\xi)\widehat{f}(\xi)),\,\,\,f\in C^\infty(G),\,x\in G.  
\end{equation*}
We will use the notation    $\widehat{ \mathcal{M}}$ for  the matrix-valued symbol of the operator $\mathcal{M}:=(1+\mathcal{L})^{\frac{1}{2}},$ and,  for every $[\xi]\in \widehat{G}$ and $s\in \mathbb{R},$ we define
   \begin{equation*}
       \widehat{ \mathcal{M}}(\xi)^{s}:=\textnormal{diag}[(1+\nu_{ii}(\xi)^2)^{\frac{s}{2}}]_{1\leqslant i\leqslant d_\xi},
   \end{equation*} where $\widehat{\mathcal{L}}(\xi)=:\textnormal{diag}[\nu_{ii}(\xi)^2]_{1\leqslant i\leqslant d_\xi}$ is the symbol of the sub-Laplacian $\mathcal{L}$ at $[\xi].$

\begin{definition}[Subelliptic amplitudes] A function $a:G\times G\times \widehat{G}\rightarrow \cup_{[\xi]\in \widehat{G}}\mathbb{C}^{d_\xi\times d_\xi}$ is an amplitude symbol if for every $[\xi]\in \widehat{G},$ $a(\cdot,\cdot,[\xi])$ is smooth. In addition, $a$ belongs to the subelliptic amplitude class of order $m$ and type $(\rho,\delta)$, $\mathcal{A}^{m,\,\mathcal{L}}_{\rho,\delta}(G\times G\times \widehat{G})$ if
   \begin{equation}\label{InIC2}
     \sup_{(x,y,[\xi])\in G\times G\times  \widehat{G} }\Vert \widehat{ \mathcal{M}}(\xi)^{(\rho|\alpha|-\delta(|\beta|+|\gamma|)-m)}\partial_{X}^{(\beta)}\partial_{Y}^{(\gamma)} \Delta_{\xi}^{\alpha}a(x,y,\xi)\Vert_{\textnormal{op}} <\infty,
   \end{equation}
   and 
   \begin{equation}\label{InIIC2}
      \sup_{(x,y,[\xi])\in G\times G\times  \widehat{G} }\Vert (\partial_{X}^{(\beta)}\partial_{Y}^{(\gamma)} \Delta_{\xi}^{\alpha} a(x,y,\xi) ) \widehat{ \mathcal{M}}(\xi)^{(\rho|\alpha|-\delta(|\beta|+|\gamma|)-m)}\Vert_{\textnormal{op}} <\infty.
   \end{equation}
  The amplitude operator associated with an amplitude $a\in \mathcal{A}^{m,\,\mathcal{L}}_{\rho,\delta}(G\times G\times \widehat{G})$ is defined via
\begin{equation*}
    Af(x)\equiv \textnormal{AOp}(a)f(x):=\sum_{[\xi]\in \widehat{G}}d_\xi \textnormal{{Tr}}\left(\xi(x)\smallint\limits_{G} a(x,y,\xi)\xi(y)^{*}f(y)dy\right),  
\end{equation*}where $f\in C^\infty(G).$
\end{definition}
The decay properties of subelliptic symbols are summarized in the following lemma (see \cite[Chapter 4]{CardonaRuzhanskyC}).

\begin{lemma}\label{lemadecaying1}
Let $G$ be a compact Lie group and  let $0\leqslant \delta,\rho\leqslant 1.$ If $a\in {S}^{m,\,\mathcal{L}}_{\rho,\delta}(G\times \widehat{G}),$ then for every $\alpha,\beta\in \mathbb{N}_0^n,$ there exists $C_{\alpha,\beta}>0$ satisfying the estimates
\begin{equation*}
    \Vert \partial_{X}^{(\beta)} \Delta_{\xi}^{\alpha}a(x,\xi)\Vert_{\textnormal{op}}\leqslant C_{\alpha,\beta}\sup_{1\leqslant i\leqslant d_\xi}(1+\nu_{ii}(\xi)^2)^{\frac{m-\rho|\alpha|+\delta|\beta|}{2 }},
\end{equation*}uniformly in $(x,[\xi])\in G\times \widehat{G}.$ 
\end{lemma}

In the next theorem we describe some fundamental properties of the subelliptic calculus \cite{CardonaRuzhanskyC}.
\begin{theorem}\label{calculus} Let $0\leqslant \delta<\rho\leqslant 1,$ and let  $\Psi^{m,\,\mathcal{L}}_{\rho,\delta}:=\textnormal{Op}({S}^{m,\,\mathcal{L}}_{\rho,\delta}(G\times \widehat{G})),$ for every $m\in \mathbb{R}.$ Then,
\begin{itemize}
    \item [-] The mapping $A\mapsto A^{*}:\Psi^{m,\,\mathcal{L}}_{\rho,\delta}\rightarrow \Psi^{m,\,\mathcal{L}}_{\rho,\delta}$ is a continuous linear mapping between Fr\'echet spaces and  the  symbol of $A^*,$ $\sigma_{A^*}(x,\xi)$ satisfies the asymptotic expansion,
 \begin{equation*}
    \widehat{A^{*}}(x,\xi)\sim \sum_{|\alpha|= 0}^\infty\Delta_{\xi}^\alpha\partial_{X}^{(\alpha)} (\widehat{A}(x,\xi)^{*}).
 \end{equation*} This means that, for every $N\in \mathbb{N},$ and for all $\ell\in \mathbb{N},$
\begin{equation*}
   \Small{ \Delta_{\xi}^{\alpha_\ell}\partial_{X}^{(\beta)}\left(\widehat{A^{*}}(x,\xi)-\sum_{|\alpha|\leqslant N}\Delta_{\xi}^\alpha\partial_{X}^{(\alpha)} (\widehat{A}(x,\xi)^{*}) \right)\in {S}^{m-(\rho-\delta)(N+1)-\rho\ell+\delta|\beta|,\,\mathcal{L}}_{\rho,\delta}(G\times\widehat{G}) },
\end{equation*} where $|\alpha_\ell|=\ell.$
\item [-] The mapping $(A_1,A_2)\mapsto A_1\circ A_2: \Psi^{m_1,\,\mathcal{L}}_{\rho,\delta}\times \Psi^{m_2,\,\mathcal{L}}_{\rho,\delta}\rightarrow \Psi^{m_3,\,\mathcal{L}}_{\rho,\delta}$ is a continuous bilinear mapping between Fr\'echet spaces, and the symbol of $A=A_{1}\circ A_2$ satisfies is given by the asymptotic formula
\begin{equation*}
    \sigma_A(x,\xi)\sim \sum_{|\alpha|= 0}^\infty(\Delta_{\xi}^\alpha\widehat{A}_{1}(x,\xi))(\partial_{X}^{(\alpha)} \widehat{A}_2(x,\xi)),
\end{equation*}which, in particular, means that, for every $N\in \mathbb{N},$ and for all $\ell \in\mathbb{N},$
\begin{align*}
    &\Delta_{\xi}^{\alpha_\ell}\partial_{X}^{(\beta)}\left(\sigma_A(x,\xi)-\sum_{|\alpha|\leqslant N}  (\Delta_{\xi}^\alpha\widehat{A}_{1}(x,\xi))(\partial_{X}^{(\alpha)} \widehat{A}_2(x,\xi))  \right)\\
    &\hspace{2cm}\in {S}^{m_1+m_2-(\rho-\delta)(N+1)-\rho\ell+\delta|\beta|,\,\mathcal{L}}_{\rho,\delta}(G\times \widehat{G}),
\end{align*}for all  $\alpha_\ell \in \mathbb{N}_0^n$ with $|\alpha_\ell|=\ell.$
\item [-] For  $0\leqslant \delta< \rho\leqslant    1,$  (or for $0\leq \delta\leq \rho\leq 1,$  $\delta<1/\kappa$) let us consider a continuous linear operator $A:C^\infty(G)\rightarrow\mathscr{D}'(G)$ with symbol  $\sigma\in {S}^{0,\,\mathcal{L}}_{\rho,\delta}(G\times \widehat{G})$. Then $A$ extends to a bounded operator from $L^2(G)$ to  $L^2(G).$ 
\end{itemize}
\end{theorem}

Here and in the rest of the paper we shall denote by  $H^{s,\,\mathcal{L}}(G)$, for $s\in \mathbb{R}$, the subelliptic Sobolev space of order $s$ associated with a fixed positive sub-Laplacian $\mathcal{L}$, that is, the space defined as the completion of $C^\infty(G)$ with respect to the norm $$ \Vert u\Vert_{H^{s,\,\mathcal{L}}(G)}:=\Vert (1+\mathcal{L})^{\frac{s}{2}}u\Vert_{L^2(G)}.$$

\begin{remark}
The last assertion in Theorem \ref{calculus} remains valid if we consider $0\leq \delta\leq \rho\leq 1,$ $\delta<1/\kappa.$ This is  the subelliptic Calder\'on-Vaillancourt theorem proved in \cite{CardonaRuzhanskyC},  which gives the boundedness  of pseudo-differential operators in the subelliptic calculus in subelliptic Sobolev spaces.
\end{remark}

\begin{proposition}
Let $A:C^{\infty}(G)\rightarrow \mathscr{D}'(G)$ be a continuous linear operator with symbol $a\in S^{m,\,\mathcal{L}}_{\rho,\delta}(G\times \widehat{G}),$  $0\leq \delta< \rho\leq 1.$ Then $A:H^{s,\,\mathcal{L}}(G)\rightarrow H^{s-m,\,\mathcal{L}}(G) $ extends to a bounded operator for all $s\in \mathbb{R}.$
\end{proposition}

Next, we state the G\r{a}rding inequality and the Sharp G\r{a}rding inequality proved for the  subelliptic  H\"ormander classes in \cite{CardonaRuzhanskyC,SharpGarding}.
\begin{theorem}[Subelliptic G\r{a}rding inequality]\label{GardinTheorem}  For $0\leqslant \delta<\rho\leqslant 1,$  let $a(x,D):C^\infty(G)\rightarrow\mathscr{D}'(G)$ be an operator with symbol  $a\in {S}^{m,\mathcal{L}}_{\rho,\delta}( G\times \widehat{G})$, $m\in \mathbb{R}$. Let us assume that
\begin{equation}
     A(x,\xi)=\frac{1}{2}(a(x,\xi)+a(x,\xi)^{*})\geqslant C\widehat{\mathcal{M}}(\xi)^{m},
\end{equation}for all $(x,[\xi])\in G\times \widehat{G},$ for some $C>0.$
 Then, there exist $C_{1},C_{2}>0,$ such that the lower bound
\begin{align}
     \mathsf{Re}(a(x,D)u,u) \geqslant C_1\Vert u\Vert_{{L}^{2,\mathcal{L}}_{\frac{m}{2}}(G)}-C_2\Vert u\Vert_{L^2(G)}^2,
\end{align}holds true for every $u\in C^\infty(G).$
\end{theorem}Theorem \ref{GardinTheorem} motivates to define the  sub-family of $\mathcal{L}$-elliptic operators, namely, the class of strongly $\mathcal{L}$-elliptic operators.
\begin{definition}[Strongly $\mathcal{L}$-elliptic operators] We will say that an operator $A\in \Psi^{m,\mathcal{L}}_{\rho,\delta}(G\times\widehat{G})$ with $0\leq \delta<\rho\leq 1,$ is strongly  $\mathcal{L}$-elliptic or order $m$ if there exists $C>0,$ such that 
\begin{equation}
     \mathsf{Re}(a(x,\xi))=\frac{1}{2}(a(x,\xi)+a(x,\xi)^{*})\geqslant C\widehat{\mathcal{M}}(\xi)^{m},
\end{equation}for all $(x,[\xi])\in G\times \widehat{G}.$
\end{definition}
Now, we are going to present the corresponding lower bound for subelliptic operators with non-negative symbols on any representation space, as in \cite{SharpGarding}.
\begin{theorem}[Subelliptic sharp G\r{a}rding inequality]\label{MainTheorem}  For $0< \rho\leqslant 1$ and  $0\leq \delta<(2\kappa-1)^{-1}\rho$, and for $m\in \mathbb{R}$, let  $$A\equiv a(x,D):C^\infty(G)\rightarrow\mathscr{D}'(G)$$ be a pseudo-differential operator with global symbol  $a\in {S}^{m,\,\mathcal{L}}_{\rho,\delta}( G\times \widehat{G})$. Then, if $a(x,[\xi])\geq 0$ for all $(x,[\xi])\in G\times \widehat{G}$, there exists a positive constant $C$ such that
\begin{equation}\label{SGIT}
    \mathsf{Re}(Au,u)\geq -C\Vert u\Vert_{H^{\frac{m-\varkappa}{2},\,\mathcal{L}}(G)}^2,\quad \varkappa  ={\rho}/{\kappa}-\left(2-{1}/{\kappa}\right)\delta,
\end{equation}
for all $u\in C^{\infty}(G).$
\end{theorem}

\section{Analysis of diffusion problems on compact Lie groups}\label{GST}
Let $T>0.$ In this section we will study the regularity, existence and uniqueness  for the  Cauchy problem 
\begin{equation}\label{PVI}(\textnormal{IVP}): \begin{cases}\frac{\partial v}{\partial t}=K(t)v+f ,& \text{ }v\in \mathscr{D}'((0,T)\times G)
\\v(0)=u_0 ,& \text{ } \end{cases}
\end{equation}   associated to a (possibly) time-dependent  continuous family of pseudo-differential operators $K(t).$ According to the literature, we start by considering the general case where  $K(t)\in \Psi^{m,\mathcal{L}}_{\rho,\delta}(G\times \widehat{G})$ belongs to the subelliptic H\"ormander classes on $G,$  with the initial data $u_0\in H^{s,\mathcal{L}}(G),$ and  $f\in  L^2([0,T],H^{s,\mathcal{L}}(G)) ,$ where  $m>0,$ and $s\in \mathbb{R}.$ 

The generality of this consideration is justified by the consideration of  (possibly) continuous time-dependent family of subellitic pseudo-differential operators $K(t)$ on $G$ with a fixed positive order $m$.  We split our analysis in two cases.
\begin{itemize}
    \item Case 1. We consider that a family of continuous time-dependent subellitic pseudo-differential operators $K(t)$ is strongly $\mathcal{L}$-elliptic, that is, for any $0\leqslant t\leqslant T,$ the estimate
    \begin{equation}
  \sigma_{ \mathsf{Re}(-K(t))}(x,[\xi])\geq \widehat{\mathcal{M}}(\xi)^m,\,\,(x,[\xi])\in G\times \widehat{G},
\end{equation} holds on any representation space.  In this case, the G\r{a}rding inequality for the subelliptic classes $$\Psi^{m,\mathcal{L}}_{\rho,\delta}(G\times \widehat{G})$$ allows us  to consider a general non-negative order  $m>0.$
    \item Case 2. The time-dependent family of symbols associated to the family of continuous time-dependent subellitic pseudo-differential operators $K(t)$ is negative definite on every representation space and for any $0\leqslant t\leqslant T,$ that is,
    \begin{equation}
   \sigma_{ \mathsf{Re}(-K(t))}(x,[\xi])\geq 0,\,\,(x,[\xi])\in G\times \widehat{G}.
\end{equation}In this case, the sharp-G\r{a}rding inequality  for the subelliptic classes $$\Psi^{m,\mathcal{L}}_{\rho,\delta}(G\times \widehat{G})$$ allows us  to analyse the strictly fractional case 
\begin{equation}\label{fractional:index}
    0\leq m\leq  \varkappa  ={\rho}/{\kappa}-\left(2-{1}/{\kappa}\right)\delta.
\end{equation}
In the case where the family of operators $K(t)$ is contained in the elliptic H\"ormander classes $\Psi^{m}_{\rho,\delta}(G\times \widehat{G}),$ the order restriction becomes to be $$0\leq m\leq \rho-\delta,$$ for all $0\leq \delta<\rho\leq 1,$ which recaptures \eqref{fractional:index} for $\kappa=1.$  
    
\end{itemize}

The next subsection will be dedicated to stating the technical results in the analysis of the well-posedness of \eqref{PVI}.
\subsection{Preliminary results} One of the facts of central difference between the subelliptic pseudo-differential calculus in \cite{CardonaRuzhanskyC}, and then of the elliptic calculus in \cite{Ruz} with the H\"ormander calculus via local coordinate systems is that, in general, the order of the commutator of two operators in both calculus does not decrease, which is a natural consequence of the non-commutativity of the products of arbitrary matrices. 

This point makes  the techniques under which the theory of initial value problems of pseudo-differential operators in Taylor \cite{Taylorbook1981} based on commutator techniques (and several of the argument there involving commutators of pseudo-differential operators) inoperable in the setting of the subelliptic classes $\Psi^{m,\mathcal{L}}_{\rho,\delta}(G\times \widehat{G}).$ The following lemma will be useful in the proof of the well posedness of  \eqref{PVI}  without the use of commutator techniques.
\begin{lemma}\label{lemenr1} Let $A\in \Psi^{\nu,\mathcal{L}}_{\rho,\delta}(G\times \widehat{G})$ be a $\mathcal{L}$-elliptic operator for some  $\nu\in\mathbb{R}$. 
Then,  for all $s\in\mathbb{R}$ there exists a constant $C>0$ such that
$$\|A\mathcal{M}_sv\|_{L^2(G)}\leq C[\|Av\|_{H^{s,\mathcal{L}}(G)}+\|v\|_{H^{s,\mathcal{L}}(G)}],$$
for all $v\in H^{s,\mathcal{L}}(G)$. Here $\mathcal{M}_s:=(1+\mathcal{L})^{\frac{s}{2}}$ is the subelliptic Bessel potential of order $s\in \mathbb{R}.$
\end{lemma}
\begin{proof} Since $A\mathcal{M}_s\in \Psi_{\rho,\delta}^{\nu+s,\mathcal{L}}(G\times \widehat{G}),$ in view of  its Sobolev boundedness,  there exists a constant $C_1>0$ such that 
\begin{equation}\label{rt67}
    \|A\mathcal{M}_sv\|_{L^2(G)}\leq C_1\|v\|_{H^{\nu+s,\mathcal{L}}(G)}
\end{equation}
for all $v\in H^{\nu+s,\mathcal{L}}(G)$. 

By the $\mathcal{L}$-ellipticity of   $\mathcal{M}_sA$, there exists a subelliptic parametrix $T\in \Psi_{\rho,\delta}^{-(\nu+s),\mathcal{L}}(G\times \widehat{G}) $ of  
 $\mathcal{M}_sA$. Then, there exists a constant $C_2>0$ such that   
\begin{equation}\label{bouelT}
    \|Tu\|_{H^{\nu+s,\mathcal{L}}(G)}\leq C_2\|u\|_{L^2(G)},\,\,\textnormal{ for all }u\in L^2(G).
\end{equation}

On the other hand, let $R=R(x,D)\in \Psi_{\rho,\delta}^{-\infty,\mathcal{L}}(G\times\widehat{G})$ be such that 
 $$T\mathcal{M}_sA=I+R(x,D),$$
and let $v\in  H^{\nu+s,\mathcal{L}}(G)$ and $u=\mathcal{M}_sAv$. Then $u\in L^2(G)$ and 
$$Tu=T\mathcal{M}_sAv=v+R(x,D)v.$$
Hence and by \eqref{bouelT} we have
$$\|v\|_{H^{\nu+s,\mathcal{L}}(G)}-\|R(x,D)v\|_{H^{\nu+s,\mathcal{L}}(G)}\leq \|v+R(x,D)v\|_{H^{\nu+s,\mathcal{L}}(G)}=\Vert T\mathcal{M}_sAu\Vert_{H^{\nu+s,\mathcal{L}}(G)}$$
$$\leq C_2\|\mathcal{M}_sAv\|_{L^2(G)}.$$
Thus
$$\|v\|_{H^{\nu+s,\mathcal{L}}(G)}\leq C_2\|\mathcal{M}_sAv\|_{L^2(G)}+\|R(x,D)v\|_{H^{\nu+s,\mathcal{L}}(G)}.$$
Now, since $R(x,D)$ belongs in particular to $\Psi_{\rho,\delta}^{-\nu,\mathcal{L}}(G\times \widehat{G})$, we have  
$$\|R(x,D)v\|_{H^{\nu+s,\mathcal{L}}(G)}\leq C_3\|v\|_{H^{s,\mathcal{L}}(G)}.$$
Then 
\begin{equation}\label{utop0}
  \|v\|_{H^{\nu+s,\mathcal{L}}(G)}\leq C_2\|\mathcal{M}_sAv\|_{L^2(G)}+\|v\|_{H^{s,\mathcal{L}}(G)}.  
\end{equation}

Therefore, by \eqref{rt67} and \eqref{utop0} we obtain
$$\|A\mathcal{M}_sv\|_{L^2(G)}\leq C(\|\mathcal{M}_sAv\|_{L^2(G)}+\|v\|_{H^{s}(G)}),$$
which concludes the proof of the lemma.
\end{proof}
In the following lemma we prove an energy identity that allows us to apply the lower bounds for the G\r{a}rding inequality and the sharp G\r{a}rding inequality in the form of  Theorems \ref{GardinTheorem} and \ref{MainTheorem}.

\begin{lemma}\label{Lemma:Energy} Let $G$ be a compact  Lie group. Let us assume that $f\in C([0,T],\mathscr{D}(G))$ and $v\in C^{1}([0,T],\mathscr{D}(G))$  are distributions such that 
$$ \frac{\partial v}{\partial t}=K(t)v+f  $$
where, for any $t\in [0,T],$ $K(t):C([0,T],\mathscr{D}(G))\rightarrow C([0,T],\mathscr{D}(G))$ is a continuous linear operator with respect to  the standard Fr\'echet structure on  $C([0,T],\mathscr{D}(G)).$ Then, 
\begin{equation}\label{Ide:Eq}
    \frac{d}{dt}\Vert v(t) \Vert^2_{L^2(G)}=2 \mathsf{Re}(K(t)v(t),v(t))_{L^2(G)}+2 \mathsf{Re}(f(t), v(t))_{L^2(G)},
\end{equation}for any $t\in [0,T].$
\end{lemma}
\begin{proof}Indeed, 
observe that
\begin{align*}
   \frac{d}{dt}\Vert v(t) \Vert^2_{L^2(G)}&= \frac{d}{dt}\left(v(t),v(t)\right)_{L^2(G)}=\left(\frac{d v(t)}{dt},v(t)\right)_{L^2(G)}+\left(v(t),\frac{d v(t)}{dt}\right)_{L^2(G)}\\
   &=\left(K(t)v(t)+f(t),v(t)\right)_{L^2(G)}+\left(v(t),K(t)v(t)+f(t)\right)_{L^2(G)}\\
    &=\left((K(t)+K(t)^{*})v(t),v(t)\right)_{L^2(G)}+2 \mathsf{Re}(f(t), v(t))_{L^2(G)}\\
     &=2 \mathsf{Re}(K(t)v(t),v(t))_{L^2(G)}+2 \mathsf{Re}(f(t), v(t))_{L^2(G)},
\end{align*}for any $t\in [0,T].$ The proof of the identity  \eqref{Ide:Eq} is complete.
\end{proof}
The next lemma proves that the energy estimates on subelliptic Sobolev spaces imply the existence and uniqueness of the IVP problem \eqref{PVI:2} via the Hanh-Banach theorem.
\begin{lemma}\label{Hanh:Banach:Lemma}
Let $s\in \mathbb{R}$ and let $\mathbb{X}=C^1([0,T],H^{-s,\mathcal{L}}(G)).$ Assume that for any $v\in \mathbb{X}$ the following energy estimate
\begin{equation}\label{Condition:on:X}
\Vert v(t) \Vert^2_{H^{-s,\mathcal{L}}(G)}\leqslant   \left(C\Vert v(0) \Vert^2_{H^{-s,\mathcal{L}}(G)}+C'\smallint\limits_{0}^T \Vert (\partial_{t}-K(\tau))v(\tau) \Vert^2_{H^{-s,\mathcal{L}}(G)}d\tau \right),  
\end{equation} holds for every $0\leqslant t\leqslant T,$ where $K(t)\in C([0,T],\Psi^{m,\mathcal{L}}_{\rho,\delta}(G\times \widehat{G})).$ Then, for any    $f\in C([0,T],\mathscr{D}(G))$ and $u_0\in C^{1}([0,T],\mathscr{D}(G))$  there exists a unique $v\in L^2([0,T],H^{s,\mathcal{L}}(G))$ solving the following initial value problem
\begin{equation}\label{PVI:2}(\textnormal{IVP}): \begin{cases}\frac{\partial v}{\partial t}=K(t)v+f ,& \text{ }v\in \mathscr{D}'((0,T)\times G)
\\v(0)=u_0 ,& \text{ } \end{cases}
\end{equation}and  satisfying the energy inequality 
\begin{equation}
\Vert v(t) \Vert^2_{H^{s,\mathcal{L}}(G)}\leqslant   \left(C\Vert u_0 \Vert^2_{H^{s,\mathcal{L}}(G)}+C'\smallint\limits_{0}^T \Vert f(\tau) \Vert^2_{H^{s,\mathcal{L}}(G)}d\tau \right),  
\end{equation} for every $0\leqslant t\leqslant T,$ between the solutions and the initial data.

\end{lemma}
\begin{proof}
 Now, we will use the Hanh-Banach theorem to deduce that the solution to \eqref{PVI:2} indeed, belongs to $\mathbb{X}
$.  So, let us denote  $Q=\partial_{t}-K(t),$ and let us introduce the spaces
    \begin{align*}
        E:=\{\phi\in \mathbb{X} :\phi(T)=0\},
    \end{align*}
    and $ Q^*E:=\{Q^*\phi:\phi\in E\}.$ Let us define the linear form $\beta\in (Q^*E)'$ by
    \begin{equation*}
    \beta(Q^*\phi):=\smallint\limits_{0}^T(f(\tau),\phi(\tau))d\tau+\frac{1}{i}(u_0,\phi(0)).
    \end{equation*}
    From \eqref{Condition:on:X} we have that for every $\phi\in E,$
    \begin{align*}
    \Vert \phi(t) \Vert^2_{H^{-s,\mathcal{L}}(G)}&\leqslant   \left(C\Vert \phi(T) \Vert^2_{H^{-s,\mathcal{L}}(G)}+C'\smallint\limits_{0}^T \Vert (\partial_{t}-K(\tau)^{*})\phi(\tau) \Vert^2_{H^{-s,\mathcal{L}}(G)}d\tau \right)\\
    &=C'\smallint\limits_{0}^T \Vert (\partial_{t}-K(\tau))^{*}\phi(\tau) \Vert^2_{H^{-s,\mathcal{L}}(G)}d\tau\\
    &=C' \Vert (\partial_{t}-K(\tau))^{*}\phi(\tau) \Vert^2_{L^2([0,T],H^{-s,\mathcal{L}}(G))}.
    \end{align*}
    The previous analysis for implies the inequality
    \begin{equation}\label{estimate;question:michael}
         \Vert \phi(t) \Vert^2_{H^{-s,\mathcal{L}}}\lesssim \Vert (\partial_{t}-K(\tau))^{*}\phi(\tau) \Vert^2_{L^2([0,T],H^{-s,\mathcal{L}})}=\Vert Q^{*}\phi(\tau) \Vert^2_{L^2([0,T],H^{-s,\mathcal{L}})}.
    \end{equation}
 {{So, in view of \eqref{estimate;question:michael}  we have}}
    \begin{align*}
    |\beta(Q^*\phi)|&\leqslant \smallint\limits_{0}^T\Vert f(\tau)\Vert_{H^{s,\mathcal{L}}(G)}\Vert \phi(\tau)\Vert_{H^{-s,\mathcal{L}}(G)} d\tau+\Vert u_0\Vert_{{H^{s,\mathcal{L}}(G)}}\Vert \phi(0)\Vert_{H^{-s,\mathcal{L}}(G)}\\
    &\leqslant \Vert f\Vert_{L^2([0,T],H^{s,\mathcal{L}}(G))}\Vert \phi\Vert_{L^2([0,T],H^{-s,\mathcal{L}}(G))}+\Vert u_0\Vert_{H^{s,\mathcal{L}}(G)}\Vert \phi\Vert_{H^{-s,\mathcal{L}}(G)}\\
    &\lesssim_{T} (\Vert f\Vert_{L^2([0,T],H^{s,\mathcal{L}}(G)}+\Vert u_0\Vert_{H^{s,\mathcal{L}}(G)})\Vert Q^*\phi(\tau) \Vert_{L^2([0,T],H^{-s,\mathcal{L}}(G))},
    \end{align*}which shows that $\beta$ is a bounded functional on $$ \mathscr{T}:=Q^*E\cap L^2([0,T],H^{-s,\mathcal{L}}(G)),$$ with $\mathscr{T}$ endowed with the topology induced by the  norm of $L^2([0,T],H^{-s,\mathcal{L}}(G)).$ By using the Hahn-Banach extension theorem, we can extends $\beta$ to a bounded functional $\tilde{\beta}$ on $L^2([0,T],H^{-s,\mathcal{L}}(G)),$ and by using the Riesz representation theorem, there exists $v\in (L^2([0,T],H^{-s,\mathcal{L}}(G))'=L^2([0,T],H^{s,\mathcal{L}}(G)) ,$ such that
    \begin{equation*}
      \tilde{\beta}(\psi)=(v,\psi),\quad \psi\in   L^2([0,T],H^{-s,\mathcal{L}}(G)).
    \end{equation*}In particular, for $\psi=Q^*\phi\in \mathscr{T},$ we have
    \begin{equation*}
      \tilde{\beta}(Q^* \phi)={\beta}(Q^*\phi)=(v,Q^*\phi),
    \end{equation*}  Because, we can identify $C^{\infty}_{0}((0,T),\mathscr{D}(G))$ as a subspace of $E$ 
    \begin{equation*}
       C^{\infty}_{0}((0,T),C^\infty(G))\subset E=\{\phi\in  C^1([0,T], H^{s,\mathcal{L}}(G) )   :\phi(T)=0\},
    \end{equation*} we have the identity
    \begin{align*}
      (f,\phi)=  \smallint\limits_{0}^T(f(\tau),\phi(\tau))d\tau= \smallint\limits_{0}^T(f(\tau),\phi(\tau))d\tau+\frac{1}{i}(u_0,\phi(0))=(v,Q^*\phi),
    \end{align*}for every $\phi \in C^{\infty}_{0}((0,T),C^\infty(G)). $ So, this implies that $v\in \textnormal{Dom}(Q^{**}).$ Because $Q^{**}=Q,$ we have that
    \begin{align*}
        (v,Q^*\phi)=(Qv,\phi)=(f,\phi),\,\,\forall\phi \in C^{\infty}_{0}((0,T),C^\infty(G)),
    \end{align*}which implies that $Qv=f.$ A routine argument of integration by parts shows that $v(0)=u_0.$ Now, in order to show the uniqueness of $v,$ let us assume that $u\in \mathbb{X}$ is a solution of the problem
    \begin{equation*} \begin{cases}\frac{\partial u}{\partial t}=K(t,x,D)u+f ,& \text{ }u\in \mathscr{D}'((0,T)\times G)
\\u(0)=u_0 .& \text{ } \end{cases}
\end{equation*} Then $\omega:=v-u\in \mathbb{X}$ solves the problem
\begin{equation*} \begin{cases}\frac{\partial \omega}{\partial t}=K(t,x,D)\omega,& \text{ }\omega\in \mathscr{D}'((0,T)\times G)
\\\omega(0)=0 ,& \text{ } \end{cases}
\end{equation*}and from Proposition \ref{energyestimate}, $\Vert \omega(t)\Vert_{H^{s,\mathcal{L}}(G)}=0,$ for all $0\leqslant t\leqslant T,$ and consequently, from the continuity in $t$ of the functions we have that $v(x,t)=u(x,t)$ for all $t\in [0,T]$ and a.e. $x\in G.$
\end{proof}

We will make use of the following lemma (see Lemma 4.4 of \cite[Chapter 6]{chpi:book}). Indeed, in view of Lemma \ref{Lemma:TheBook} the terms $y(\tau)=\Vert v(\tau)\Vert_{H^{s,\mathcal{L}}}$ inside of the integrals in our energy estimates can be neglected (compare e.g. \eqref{partialII} with \eqref{partialIII}). 
\begin{lemma}\label{Lemma:TheBook}Suppose that $y:[0,T]\rightarrow \mathbb{R}^+_0$ and $\phi:[0,T]\rightarrow \mathbb{R}^+_0$ are continuous and positive functions such that
\begin{equation}
  \forall t\in [0,T],\,  y(t)\leq C\left(\alpha+\smallint\limits_{0}^{T}[y(\tau)+\phi(\tau)]d\tau\right).
\end{equation}Then,
\begin{equation}
  \forall t\in [0,T],\,   y(t)\leq C_T \left(\alpha e^{Ct}+\smallint\limits_{0}^{T}e^{C(t-\tau)}\phi(\tau)d\tau\right).
\end{equation}
\end{lemma}

\subsection{Case I: Diffusion problems for strongly elliptic operators with positive order}\label{subsectio:case:1}

Now, we are ready for the proof of our energy estimate when $K(t)$ is strongly elliptic for all $0\leq t\leq T.$

\begin{proposition}\label{energyestimate:apriori} Let  $0\leq \delta< \rho\leqslant 1,$ and let $K(t)\in C([0,T],\Psi^{m,\mathcal{L}}_{\rho,\delta}(G\times \widehat{G}))$   be a continuous family of  pseudo-differential operators of order $m>0.$ Assume that for all $t\in [0,T],$ $K(t)$ is strongly $\mathcal{L}$-elliptic.
  If $v\in C^1([0,T],H^s(G))\cap C([0,T],H^{s+\frac{m}{2},\mathcal{L}}(G))$ then there exist $C,C'>0,$ such that
\begin{equation}
\Vert v(t) \Vert^2_{H^{s,\mathcal{L}}(G)}\leqslant   \left(C\Vert v(0) \Vert^2_{H^{s,\mathcal{L}}(G)}+C'\smallint\limits_{0}^T \Vert (\partial_{t}-K(\tau))v(\tau) \Vert^2_{H^{s,\mathcal{L}}(G)}d\tau \right),  
\end{equation}holds for every $0\leqslant t\leqslant T.$ Moreover, we also have the estimate
\begin{equation}
\Vert v(t) \Vert^2_{H^{s,\mathcal{L}}(G)}\leqslant   \left(C\Vert v(T) \Vert^2_{H^{s,\mathcal{L}}(G)}+C'\smallint\limits_{0}^T \Vert (\partial_{t}-K(\tau)^{*})v(\tau) \Vert^2_{H^{s,\mathcal{L}}(G)}d\tau \right).  
\end{equation}
\end{proposition}
\begin{proof} 
Let $v\in  C^1([0,T],H^s(G))\cap C([0,T],H^{s+\frac{m}{2},\mathcal{L}}(G)).$  This fact will be useful later because we will use the  G\r{a}rding inequality in Theorem \ref{GardinTheorem} applied to the operator $ \mathsf{Re}(-K(t)).$ For instance we will use that for $0\leq t\leq T,$ $u_{t}\in L^2(G).$ So, 
$v\in \textnormal{Dom}(\partial_{\tau}-K(\tau))$ for every $0\leqslant \tau\leqslant T.$ 

Let us define $f(\tau):=Q(\tau)v(\tau),$ with $Q(\tau):=(\partial_{\tau}-K(\tau)),$ for every $0\leqslant \tau\leqslant T.$ 
Using Lemma \ref{Lemma:Energy} we deduce that 
\begin{equation}
    \frac{d}{dt}\Vert v(t) \Vert^2_{L^2(G)}=2 \mathsf{Re}(K(t)v(t),v(t))_{L^2(G)}+2 \mathsf{Re}(f(t), v(t))_{L^2(G)},
\end{equation}for any $t\in [0,T].$
Now, from the subelliptic G\r{a}rding inequality, 
\begin{align}
     \mathsf{Re}(-K(t)v(t),v(t)) \geqslant C_1\Vert v(t)\Vert_{H^{\frac{m}{2},\mathcal{L}}(G)}-C_2\Vert v(t)\Vert_{L^2(G)}^2,
\end{align}and  the parallelogram law, we have
\begin{align*}
 2 \mathsf{Re}(f(t), v(t))_{L^2(G)}&\leqslant 2 \mathsf{Re}(f(t), v(t))_{L^2(G)}+\Vert f(t)\Vert_{L^2(G)}^2+\Vert v(t)\Vert_{L^2(G)}^2 \\
 &= \Vert f(t)+v(t)\Vert^2\leqslant \Vert f(t)+v(t)\Vert^2+\Vert f(t)-v(t)\Vert^2 \\
&= 2\Vert f(t)\Vert^2_{L^2(G)}+2\Vert v(t)\Vert^2_{L^2(G)},
\end{align*}
and consequently
\begin{align*}
   & \frac{d}{dt}\Vert v(t) \Vert^2_{L^2(G)}\\
   &\leqslant 2\left(C_2\Vert v(t)\Vert_{L^2(G)}^2-C_1\Vert v(t)\Vert_{H^{\frac{m}{2},\mathcal{L}}(G)}\right)+2\Vert f(t)\Vert^2_{L^2(G)}+2\Vert v(t)\Vert^2_{L^2(G)}\\
   &\lesssim \Vert f(t)\Vert^2_{L^2(G)}+\Vert v(t)\Vert^2_{L^2(G)}.
\end{align*} By using Gronwall's Lemma we obtain the energy estimate
\begin{equation}
\Vert v(t) \Vert^2_{L^2(G)}\leqslant  \left(C\Vert v(0) \Vert_{L^2(G)}^2+C'\smallint\limits_{0}^T \Vert f(\tau) \Vert_{L^2(G)}^2d\tau \right),   
\end{equation}for every $0\leqslant t\leqslant T,$ and $T>0.$ By replacing $v(t)$ by $\mathcal{M}_s v(t),$ we have  for every $0\leqslant \tau\leqslant T,$ that
\begin{equation}
  \Vert \mathcal{M}_s v\Vert^2_{L^2(G)}\leqslant  \left(C\Vert \mathcal{M}_s v(0) \Vert_{L^2(G)}^2+C'\smallint\limits_{0}^T \Vert Q(\tau)\mathcal{M}_sv(\tau) \Vert_{L^2(G)}^2d\tau \right).  
\end{equation}
So, we have proved that 
\begin{equation}\label{partial:lemma}
  \Vert  v\Vert^2_{H^{s,\mathcal{L}}(G)}\leqslant  \left(C\Vert  v(0) \Vert_{H^{s,\mathcal{L}}(G)}^2+C'\smallint\limits_{0}^T \Vert Q(\tau)\mathcal{M}_sv(\tau) \Vert_{L^2(G)}^2d\tau \right).  
\end{equation}
Now, by using Lemma \ref{lemenr1} with $A=Q(\tau),$ we have that 
$$\|Q(\tau)\mathcal{M}_sv\|_{L^2(G)}\leq C[\|Q(\tau)v\|_{H^{s,\mathcal{L}}(G)}+\|v\|_{H^{s,\mathcal{L}}(G)}]=C[\|f(\tau)\|_{H^{s,\mathcal{L}}(G)}+\|v\|_{H^{s,\mathcal{L}}(G)}],$$ that indeed implies the estimate
\begin{equation}
    \|Q(\tau)\mathcal{M}_sv\|^2_{L^2(G)}\lesssim[\|f(\tau)\|^2_{H^{s,\mathcal{L}}(G)}+\|v\|^2_{H^{s,\mathcal{L}}(G)}].
\end{equation}
In view of \eqref{partial:lemma} we have for $0\leq t\leq T$ the inequality
\begin{equation}\label{partialII}
     \Vert  v(t)\Vert^2_{H^{s,\mathcal{L}}(G)}\lesssim \Vert  v(0) \Vert_{H^{s,\mathcal{L}}(G)}^2+\smallint\limits_{0}^T [\|v(\tau)\|^2_{H^{s,\mathcal{L}}(G)}+\|f(\tau)\|^2_{H^{s,\mathcal{L}}(G)}]d\tau.
\end{equation}
Using the integral Lemma \ref{Lemma:TheBook} to \eqref{partialII}, we deduce that
\begin{equation}\label{partialIII}
     \Vert  v(t)\Vert^2_{H^{s,\mathcal{L}}(G)}\lesssim_{T} \Vert  v(0) \Vert_{H^{s,\mathcal{L}}(G)}^2+\smallint\limits_{0}^T \|f(\tau)\|^2_{H^{s,\mathcal{L}}(G)}d\tau,
\end{equation}for $0\leq t\leq T.$
To finish the proof, we can change the analysis above with $v(T-\cdot)$ instead of $v(\cdot),$ $f(T-\cdot)$ instead of $f(\cdot)$ and $Q^{*}=-\partial_{t}-K(t)^{*},$ (or equivalently $Q=\partial_{t}-K(t)$ ) instead of $Q^{*}=-\partial_{t}+K(t)^{*}$ (or equivalently $Q=\partial_{t}-K(t)$) using that $ \mathsf{Re}(K(T-t)^*)= \mathsf{Re}(K(T-t))$ to deduce that
\begin{align*}
&\Vert v(T-t) \Vert^2_{H^{s,\mathcal{L}}(G)}\\
&\leqslant \left(C\Vert v(T) \Vert^2_{L^{s,\mathcal{L}}(G)}+C'\smallint\limits_{0}^{T} \Vert (-\partial_{t}+K(T-t)^{*})v(T-\tau) \Vert^2_{H^{s,\mathcal{L}}(G)}d\tau \right)\\
&=   \left(C\Vert v(T) \Vert^2_{H^{s,\mathcal{L}}(G)}+C'\smallint\limits_{0}^T \Vert (-\partial_{t}-K(t)^{*})v(s) \Vert^2_{H^{s,\mathcal{L}}(G)}ds \right).
\end{align*}So, we conclude the proof.
\end{proof}
Now, we will prove the existence and uniqueness for \eqref{PVI} when the initial data are taken in $L^2(G).$
\begin{theorem}\label{Main:Th}
 Let $0\leq \delta< \rho\leq 1$, $T>0$ and  $K(t)\in \Psi_{\rho, \delta}^{m,\mathcal{L}}(G\times \widehat{G}) $  be strongly elliptic for some  $m>0$. Let  $s\in\mathbb{R}$,  $u_0\in H^{s,\mathcal{L}}(G)$, $f\in L^2([0,T],H^{s,\mathcal{L}}(G))$ . Then, there exists a unique solution $v$ in the space $$ C^1([0,T],H^{s,\mathcal{L}}(G))\bigcap C([0,T],H^{s+\frac{m}{2},\mathcal{L}}(G))$$ of the Cauchy problem: \begin{equation}\label{IVP:th}(\textnormal{IVP}): \begin{cases}\frac{\partial v}{\partial t}=K(t)v+f ,& \text{ }v\in \mathscr{D}'((0,T)\times G),
\\v(0)=u_0 .& \text{ } \end{cases}
\end{equation} Moreover, $v$ satisfies the energy estimate
\begin{equation}
\Vert v(t) \Vert^2_{H^{s,\mathcal{L}}(G)}\leqslant  \left(C\Vert u_0 \Vert^2_{H^{s,\mathcal{L}}(G)}+C'\Vert f \Vert^2_{L^2([0,T],H^{s,\mathcal{L}}(G))} \right),
\end{equation}for every $0\leqslant t\leqslant T.$
\end{theorem}
\begin{proof}
We have proved the {\it a priori} estimate in Proposition \ref{energyestimate:apriori}  for  $$v\in \mathbb{X}:=C^1([0,T],H^{-s}(G))\bigcap C([0,T],H^{-s+\frac{m}{2}}(G))   .$$ Essentially the condition $v\in  C([0,T],H^{-s+m/2}(G)) $ was used when applying the G\r{a}rding inequality in \eqref{applying :Garding}. Then, the existence of $v\in C^1([0,T],H^{s}(G))$ solving \eqref{IVP:th} now follows from Lemma \ref{Hanh:Banach:Lemma}.
\end{proof}

\subsection{Case II:  Diffusion problems for operators with positive symbols on the phase space}\label{subsectio:case:1}

We will start with the following energy estimate. Since we will apply the Sharp G\r{a}rding inequality proved for the  subelliptic  H\"ormander classes in \cite{SharpGarding}, we will require the condition $0\leq \delta <(2\kappa-1)^{-1}\rho,$ with $0<\rho\leq 1.$
\begin{proposition}\label{energyestimate} For  $0< \rho\leqslant 1,$ let  $0\leq \delta<(2\kappa-1)^{-1}\rho.$
Let $K(t)=K(t,x,D)\in C([0,T],\Psi^{m,\mathcal{L}}_{\rho,\delta}(G\times \widehat{G}))$   be a continuous family of  pseudo-differential operators of order $m\in \mathbb{R},$ such that $$\forall\, (x,[\xi])\in G\times \widehat{G},\quad \sigma_{-\textnormal{Re }K(t)}(x,[\xi])\geq 0,\textnormal{ and }\,\quad 0\leq m\leq \varkappa  ={\rho}/{\kappa}-\left(2-{1}/{\kappa}\right)\delta. $$
  If  $v\in C^1([0,T], L^2(G) ),$ then there exist $C,C'>0,$ such that
\begin{equation}
\Vert v(t) \Vert^2_{L^2(G)}\leqslant   \left(C\Vert v(0) \Vert^2_{L^2(G)}+C'\smallint\limits_{0}^T \Vert (\partial_{t}-K(\tau))v(\tau) \Vert^2_{L^2(G)}d\tau \right),  
\end{equation}holds for every $0\leqslant t\leqslant T.$ Moreover, we also have the estimate
\begin{equation}\label{Q*}
\Vert v(t) \Vert^2_{L^2(G)}\leqslant   \left(C\Vert v(T) \Vert^2_{L^2(G)}+C'\smallint\limits_{0}^T \Vert (\partial_{t}-K(\tau)^{*})v(\tau) \Vert^2_{L^2(G)}d\tau \right).  
\end{equation}
\end{proposition}
\begin{proof} 
Let $v\in C^1([0,T], L^2(G) )) .$  This fact will be useful later because we will use the subelliptic sharp G\r{a}rding inequality applied to the operator $ \mathsf{Re}(-K(t)).$ For instance we will use that for $0\leq t\leq T,$ $u_{t}\in L^2(G).$ So, 
$v\in \textnormal{Dom}(\partial_{\tau}-K(\tau))$ for every $0\leqslant \tau\leqslant T.$ 

Let us define $f(\tau):=Q(\tau)v(\tau),$ with $Q(\tau):=(\partial_{\tau}-K(\tau)),$ for every $0\leqslant \tau\leqslant T.$ Observe that in view of Lemma \ref{Lemma:Energy} we have that 
\begin{align*}
   \frac{d}{dt}\Vert v(t) \Vert^2_{L^2(G)}=2 \mathsf{Re}(K(t)v(t),v(t))_{L^2(G)}+2 \mathsf{Re}(f(t), v(t))_{L^2(G)}.
\end{align*}Now, from the sharp  G\r{a}rding inequality \eqref{SGIT} we deduce that 
\begin{align}\label{applying :Garding}
     \mathsf{Re}(-K(t)v(t),v(t)) \leqslant C_1\Vert v(t)\Vert_{H^{\frac{m-\varkappa}{2}}(G)},\quad \varkappa  ={\rho}/{\kappa}-\left(2-{1}/{\kappa}\right)\delta,
\end{align}and in view of   the parallelogram law, we have
\begin{align*}
 2 \mathsf{Re}(f(t), v(t))_{L^2(G)}&\leqslant 2 \mathsf{Re}(f(t), v(t))_{L^2(G)}+\Vert f(t)\Vert_{L^2(G)}^2+\Vert v(t)\Vert_{L^2(G)}^2 \\
 &= \Vert f(t)+v(t)\Vert^2\leqslant \Vert f(t)+v(t)\Vert^2+\Vert f(t)-v(t)\Vert^2 \\
&= 2\Vert f(t)\Vert^2_{L^2(G)}+2\Vert v(t)\Vert^2_{L^2(G)}.
\end{align*}
Putting all the estimates above together we have proved that
\begin{align*}
   & \frac{d}{dt}\Vert v(t) \Vert^2_{L^2(G)}\leqslant 2\left(C_1\Vert v(t)\Vert_{H^{\frac{m-\varkappa}{2}}(G)}\right)+2\Vert f(t)\Vert^2_{L^2(G)}+2\Vert v(t)\Vert^2_{L^2(G)}.
\end{align*} Using that $m-\varkappa\leq 0,$ we have the embedding   $L^2(G)\hookrightarrow H^{\frac{m-\varkappa}{2}}(G),$ and consequently the comparison of norms $\Vert v(t)\Vert_{H^{\frac{m-\varkappa}{2}}} \lesssim \Vert v(t)\Vert_{L^2(G)} .$  So, we have the inequality
\begin{align*}
   \frac{d}{dt}\Vert v(t) \Vert^2_{L^2(G)}\lesssim  \Vert f(t)\Vert^2_{L^2(G)}+\Vert v(t)\Vert^2_{L^2(G)}.
\end{align*}By using Gronwall's Lemma we obtain the energy estimate
\begin{equation}
\Vert v(t) \Vert^2_{L^2(G)}\leqslant  \left(C\Vert v(0) \Vert_{L^2(G)}^2+C'\smallint\limits_{0}^T \Vert f(\tau) \Vert_{L^2(G)}^2d\tau \right),   
\end{equation}for every $0\leqslant t\leqslant T,$ and $T>0.$
To finish the proof, we can change the analysis above with $v(T-\cdot)$ instead of $v(\cdot),$ $f(T-\cdot)$ instead of $f(\cdot)$ and $Q^{*}=-\partial_{t}-K(t)^{*},$ (or equivalently $Q=\partial_{t}-K(t)$ ) instead of $Q^{*}=-\partial_{t}+K(t)^{*}$ (or equivalently $Q=\partial_{t}-K(t)$) using that $ \mathsf{Re}(K(T-t)^*)= \mathsf{Re}(K(T-t))$ to deduce that
\begin{align*}
&\Vert v(T-t) \Vert^2_{L^2(G)}\\
&\leqslant \left(C\Vert v(T) \Vert^2_{L^2(G)}+C'\smallint\limits_{0}^{T} \Vert (-\partial_{t}+K(T-t)^{*})v(T-\tau) \Vert^2_{L^2(G)}d\tau \right)\\
&=   \left(C\Vert v(T) \Vert^2_{L^2(G)}+C'\smallint\limits_{0}^T \Vert (-\partial_{t}-K(t)^{*})v(s) \Vert^2_{L^2(G)}ds \right).
\end{align*}So, we conclude the proof.
\end{proof}
Now, we will prove the existence and uniqueness for \eqref{PVI} when the initial data are taken in $L^2(G).$
\begin{theorem}\label{Main:Th:2}
For  $0< \rho\leqslant 1,$ let  $0\leq \delta<(2\kappa-1)^{-1}\rho.$
Let $K(t)=K(t,x,D)\in C([0,T],\Psi^{m,\mathcal{L}}_{\rho,\delta}(G\times \widehat{G}))$   be a continuous family of  pseudo-differential operators of order $m\in \mathbb{R},$ such that $$\forall\, (x,[\xi])\in G\times \widehat{G},\quad \sigma_{-\textnormal{Re }K(t)}(x,[\xi])\geq 0,\textnormal{ and }\,\quad 0\leq m\leq \varkappa={\rho}/{\kappa}-\left(2-{1}/{\kappa}\right)\delta. $$
Let   $u_0\in L^2(G)$, and let $f\in L^2([0,T],L^2(G)).$ Then there exists a unique $v\in C^1([0,T], L^2(G) ) $ solving \eqref{PVI}. Moreover, $v$ satisfies the energy estimate
\begin{equation}
\Vert v(t) \Vert^2_{L^2(G)}\leqslant  \left(C\Vert u_0 \Vert^2_{L^2(G)}+C'\Vert f \Vert^2_{L^2([0,T],L^2(G))} \right),
\end{equation}for every $0\leqslant t\leqslant T.$
\end{theorem}
\begin{proof}
We have proved the {\it a priori} estimate in Proposition \ref{energyestimate:apriori}  for  $$v\in \mathbb{X}:=C^1([0,T], L^2(G) )   .$$ Essentially the condition $v\in  C([0,T],L^2(G)) $ was used when applying the sharp G\r{a}rding inequality in Theorem \ref{MainTheorem}. Then, the existence of $v\in \mathbb{X}$ solving \eqref{IVP:th} now follows from Lemma \ref{Hanh:Banach:Lemma}.
\end{proof}
Now, we extend Theorem \ref{Main:Th} to the case where the initial data are considered in general Sobolev spaces.

\begin{theorem}\label{Main:Th:Sob}
For  $0< \rho\leqslant 1,$ let  $0\leq \delta<(2\kappa-1)^{-1}\rho,$ and let $s\in \mathbb{R}.$
Let $K(t)=K(t,x,D)\in C([0,T],\Psi^{m,\mathcal{L}}_{\rho,\delta}(G\times \widehat{G}))$   be a continuous family of  pseudo-differential operators of order $m\in \mathbb{R},$ such that $$\forall\, (x,[\xi])\in G\times \widehat{G},\quad \sigma_{-\textnormal{Re }\tilde{K}(t)}(x,[\xi])\geq 0, $$
where {{
$$  \tilde{K}(t)= (1+\mathcal{L})^{\frac{s}{2}}K(t) (1+\mathcal{L})^{-\frac{s}{2}} \textnormal{   and }\, 0\leq m\leq \varkappa  ={\rho}/{\kappa}-\left(2-{1}/{\kappa}\right)\delta. $$}}
Let $u_0\in H^{s, \mathcal{L}}(G)$, and let $f\in L^2([0,T],H^{s,\mathcal{L}}(G)).$ Then there exists a unique $v\in C^1([0,T], H^{s,\mathcal{L}}(G) )) $ solving \eqref{PVI}. Moreover, $v$ satisfies the energy estimate
\begin{equation}
\Vert v(t) \Vert^2_{H^{s,\mathcal{L}}(G)}\leqslant  \left(C\Vert u_0 \Vert^2_{H^{s,\mathcal{L}}(G)}+C'\Vert f \Vert^2_{L^2([0,T],H^{s,\mathcal{L}}(G))} \right),
\end{equation}for every $0\leqslant t\leqslant T.$
\end{theorem}

\begin{proof}Observe that if $v$ is a solution for the Cauchy problem
\begin{equation*} \begin{cases}\frac{\partial v}{\partial t}=K(t,x,D)v+f ,& \text{ }v\in \mathscr{D}'((0,T)\times G)
\\v(0)=u_0 ,& \text{ } \end{cases}
\end{equation*}where the initial data $u_0\in H^{s,\mathcal{L}}(G),$ $K(t):=K(t,x,D)\in \Psi^{m,\mathcal{L}}_{\rho,\delta}(G\times \widehat{G}),$ $f\in  L^2([0,T],H^s(G)) ,$ then we have
\begin{align*}
    \frac{\partial }{\partial t}(1+\mathcal{L})^{\frac{s}{2}} v &= (1+\mathcal{L})^{\frac{s}{2}} \frac{\partial v}{\partial t} =(1+\mathcal{L})^{\frac{s}{2}}K(t,x,D)v+(1+\mathcal{L})^{\frac{s}{2}}f\\
    &=\tilde{K}(t,x,D)(1+\mathcal{L})^{\frac{s}{2}}
v+(1+\mathcal{L})^{\frac{s}{2}}f,
\end{align*}showing that $v':= (1+\mathcal{L})^{\frac{s}{2}}
v\in L^2(G)$ and $f':=(1+\mathcal{L})^{\frac{s}{2}}f\in L^2([0,T],L^2(G))$ satisfy the initial boundary value problem
\begin{equation*} \begin{cases}\frac{\partial v'}{\partial t}=\tilde{K}(t,x,D)v'+f' ,& \text{ }v'\in \mathscr{D}'((0,T)\times G)
\\v'(0)=u'_0:=(1+\mathcal{L})^{\frac{s}{2}}u_0\in L^2(G). & \text{ } \end{cases}
\end{equation*}So, by applying Proposition \ref{Main:Th} to the family of operators $\tilde{K}(t)$ we get the estimate
\begin{equation}\label{aux:sec:2}
\Vert v'(t) \Vert^2_{L^2(G)}\leqslant  \left(C\Vert u'_0 \Vert^2_{L^2(G)}+C'\Vert f' \Vert^2_{L^2([0,T],L^2(G))} \right),
\end{equation}for every $0\leqslant t\leqslant T.$    Finally, observe that \eqref{aux:sec:2} is equivalent to the following energy inequality
\begin{equation}\label{aux:sec}
\Vert v(t) \Vert^2_{H^{s,\mathcal{L}}(G)}\leqslant  \left(C\Vert u_0 \Vert^2_{H^{s,\mathcal{L}}(G)}+C'\Vert f \Vert^2_{H^{s,\mathcal{L}}([0,T],L^2(G))} \right),
\end{equation}for every $0\leqslant t\leqslant T,$ proving Theorem \ref{Main:Th:Sob}.
\end{proof}

\subsection{Consequences: Well posedness for diffusion equations in elliptic H\"ormander classes}\label{GST:2}
This subsection will be dedicated to presenting the corresponding statements of the well-posedness for the diffusion problem  where $K(t)$ belongs to the global H\"ormander classes $\Psi^{m}_{\rho,\delta}(G\times \widehat{G}),$ for $0\leq \delta<\rho\leq 1$ developed in \cite{Ruz}. So, in particular, the results of this sub-section also cover the analysis when $K(t)$ belongs to the H\"ormander classes $\Psi^{m}_{\rho,\delta}(G,\textnormal{loc})$ defined by local coordinate systems, see H\"ormander \cite{Hormander1985III}. 
\subsubsection{The strongly elliptic case}We start with the well posedness of diffusion problems for strongly elliptic (and then, for elliptic) pseudo-differential operators.

The prototype of the operators covered by Corollary \ref{Cor:Lstrongly;ellip} are the ones with positive order $m>0$ and given by $K(t)=a(x,t)(1+\mathcal{L}_G)^{\frac{m}{2}},$ and $K(t)=a(x,t)\mathcal{L}_G^{\frac{m}{2}},$ with $a(x,t)$ being smooth in $(x,t)$ and positive. 
\begin{corollary}\label{Cor:Lstrongly;ellip}
 Let $0\leq \delta< \rho\leq 1$, $T>0$ and  $K(t)\in \Psi_{\rho, \delta}^m(G\times \widehat{G}) $  be strongly elliptic for some  $m>0$. Let  $s\in\mathbb{R}$,  $u_0\in H^{s}(G)$, $f\in L^2([0,T],H^{s}(G))$ . Then, there exists a unique solution $v$ in the space $ C^1([0,T],H^{s}(G))\bigcap C([0,T],H^{s+\frac{m}{2}}(G))$ of the Cauchy problem: \begin{equation}\label{IVP:th}(\textnormal{IVP}): \begin{cases}\frac{\partial v}{\partial t}=K(t)v+f ,& \text{ }v\in \mathscr{D}'((0,T)\times G),
\\v(0)=u_0 ,& \text{ } \end{cases}
\end{equation} Moreover, $v$ satisfies the energy estimate
\begin{equation}
\Vert v(t) \Vert^2_{H^{s}(G)}\leqslant  \left(C\Vert u_0 \Vert^2_{H^{s}(G)}+C'\Vert f \Vert^2_{L^2([0,T],H^{s}(G)))} \right),
\end{equation}for every $0\leqslant t\leqslant T.$
\end{corollary}
We also analyse the previous statement when $K(t)$ in Corollary \ref{Cor:Lstrongly;ellip}  is perturbed by a drift 
$$D=\sum_{j=1}^{n}a_{j}(x,t)X_j\in C^{\infty}([0,T], TG^{ \mathbb{C}}), \,\,a_{i}\in C^{\infty}(G\times [0,T],\mathbb{R}),\,X_j\in i\mathfrak{g}. $$
For this, see Subsection \ref{Drift:sub:sect}. Now, we are going to analyse the case where $-K(t)$ has a matrix-valued positive symbol and the order $m$ is inside of a interval determined by the subelliptic sharp G\r{a}rding inequality \cite{SharpGarding}. 
\subsubsection{The case of non-negative symbols}
Observe that in the case where the family of operators $K(t)\in \Psi^{m}_{\rho,\delta}(G\times\widehat{G}),$ that is, if we replace the sub-Laplacian $\mathcal{L}$ by the Laplace-Beltrami operator $\mathcal{L}_G$ in Theorem \ref{Main:Th:2}, the condition $$0\leq \delta<(2\kappa-1)^{-1}\rho$$ is less restrictive since $\kappa=1$ and we gain the complete range $ 0\leq \delta<\rho\leq 1.$ 
So, we have the following well-posedness result.
\begin{corollary}\label{elliptic:classes}
For  $0\leq \delta<\rho\leq 1,$
let $K(t)=K(t,x,D)\in C([0,T],\Psi^{m}_{\rho,\delta}(G\times \widehat{G}))$   be a continuous family of  pseudo-differential operators of order $m\in \mathbb{R},$ such that $$\forall\, (x,[\xi])\in G\times \widehat{G},\quad \sigma_{-\textnormal{Re }K(t)}(x,[\xi])\geq 0,\textnormal{ and }\,\quad 0\leq m\leq \rho-\delta. $$
Let   $u_0\in L^2(G)$, and let $f\in L^2([0,T],L^2(G)).$ Then there exists a unique $v\in C^1([0,T], L^2(G) ) $ solving \eqref{PVI}. Moreover, $v$ satisfies the energy estimate
\begin{equation}
\Vert v(t) \Vert^2_{L^2(G)}\leqslant  \left(C\Vert u_0 \Vert^2_{L^2(G)}+C'\Vert f \Vert^2_{L^2([0,T],L^2(G))} \right),
\end{equation}for every $0\leqslant t\leqslant T.$
\end{corollary}
Moreover,  we have the following extension of Corollary \ref{elliptic:classes} to Sobolev spaces as a consequence of Theorem \ref{Main:Th:2}. 

\begin{corollary}\label{Main:Th:Sob:2}
For  $0\leq \delta<\rho\leq 1,$
let $K(t)=K(t,x,D)\in C([0,T],\Psi^{m}_{\rho,\delta}(G\times \widehat{G}))$   be a continuous family of  pseudo-differential operators of order $m\in \mathbb{R},$   such that $$\forall\, (x,[\xi])\in G\times \widehat{G},\quad \sigma_{-\textnormal{Re }\tilde{K}(t)}(x,[\xi])\geq 0, $$
where {{
$$  \tilde{K}(t)= (1+\mathcal{L}_G)^{\frac{s}{2}}K(t) (1+\mathcal{L}_G)^{-\frac{s}{2}} \textnormal{   and }\, 0\leq m\leq \varkappa  ={\rho}-\delta. $$}} Let $s\in \mathbb{R},$    $u_0\in H^s(G)$, and let $f\in L^2([0,T],H^s(G)).$ Then there exists a unique $v\in C^1([0,T], H^{s}(G) )) $ solving \eqref{PVI}. Moreover, $v$ satisfies the energy estimate
\begin{equation}\label{energyëstimate:ellipt:drift}
\Vert v(t) \Vert^2_{H^{s}(G)}\leqslant  \left(C\Vert u_0 \Vert^2_{H^{s}(G)}+C'\Vert f \Vert^2_{L^2([0,T],H^{s}(G))} \right),
\end{equation}for every $0\leqslant t\leqslant T.$
\end{corollary}

\subsection{Analysis of diffusion equation for  fractional sub-Laplacians}\label{elliptic}
Now, let us use the analytic functional calculus in \cite[Chapter 8]{CardonaRuzhanskyC} to obtain a significant consequence of Corollary \ref{Cor:Lstrongly;ellip}. 
\begin{corollary}\label{subelliptic:fractiona:corollary}
Assume that $a(x,t)\in C^{\infty}(G\times [0,T],\mathbb{R}^+)$ does not have zeros and let $m>0$. Then, for any $s\in \mathbb{R}$ the diffusion problem
 \begin{equation}\label{Driff:driftLLaplam1:subelliptic}\begin{cases}\frac{\partial v}{\partial t}=a(x,t)\mathcal{L}^{\frac{m}{2}}v+f ,& \text{ }v\in \mathscr{D}'((0,T)\times G),
\\v(0)=u_0 ,& \text{ } \end{cases}
\end{equation} for initial data $u_0\in H^{s,\mathcal{L}}(G)$ and $f\in L^2([0,T],H^{s,\mathcal{L}}(G))$ admits a unique solution $v$ in the space $ C^1([0,T],H^{s,\mathcal{L}}(G))\bigcap C([0,T],H^{s+\frac{m}{2},\mathcal{L}}(G)).$ Moreover, $v$ satisfies the energy estimate
\begin{equation}
\Vert v(t) \Vert^2_{H^{s,\mathcal{L}}(G)}\leqslant  \left(C\Vert u_0 \Vert^2_{H^{s,\mathcal{L}}(G)}+C'\Vert f \Vert^2_{L^2([0,T],H^{s,\mathcal{L}}(G)))} \right),
\end{equation}for every $0\leqslant t\leqslant T.$ 
\end{corollary}
\begin{proof}In view of the subelliptic functional calculus  \cite[Chapter 8]{CardonaRuzhanskyC} we have that 
$
  \forall z\in \mathbb{C},  \mathcal{L}^{\frac{z}{2}}\in \Psi^{ \mathsf{Re}(z),\mathcal{L}}_{1,0}(G\times \widehat{G}), 
$ and in consequence for $z=m+0i>0,$ $\mathcal{L}^{\frac{m}{2}}\in \Psi^{m,\mathcal{L}}_{1,0}(G\times \widehat{G}).$ Since the smooth function $a(x,t)$ does not have zeroes, the $\mathcal{L}$-ellipticity of the operator $\mathcal{L}^{\frac{m}{2}}$ is preserved when multiplying it on the left with the function $a(x,t).$ So, the operator $K(t)=a(x,t)\mathcal{L}^{\frac{m}{2}}$ is $\mathcal{L}$-elliptic and the statement in Corollary \ref{subelliptic:fractiona:corollary} now follows from Corollary \ref{Cor:Lstrongly;ellip}.
\end{proof}

\subsection{Analysis of diffusion equations for fractional Laplacians with drift}\label{Drift:sub:sect}
We now state the well-posedness for the fractional diffusion with drift on compact Lie groups. To do so, let us consider the fractional diffusion problem 
\begin{equation}\label{Driff:drift}\begin{cases}\frac{\partial v}{\partial t}=(a(x,t)\mathcal{L}_{G}^{\frac{m}{2}}+D)v+f ,& \text{ }v\in \mathscr{D}'((0,T)\times G),
\\v(0)=u_0 ,& \text{ } \end{cases}
\end{equation} 
where  $0\leq m\leq 1.$ If $  D\equiv 0 ,$ we are in the case of a fractional diffusion equation for the fractional Laplacian. However, if  $$D=\sum_{j=1}^{n}a_{j}(x,t)X_j\in C^{\infty}([0,T], TG^{ \mathbb{C}}), \,\,a_{i}\in C^{\infty}(G\times [0,T],\mathbb{R}),\,X_j\in i\mathfrak{g}, $$
 is a differential operator of first order we are in the case of a diffusion equation with drift $D$. We have denoted by $TG^{ \mathbb{C}}$ to the complexification of the tangent space $$TG^{ \mathbb{C}}\cong TG\oplus i TG\cong (G\times G)\times (\mathfrak{g}\otimes i\mathfrak{g}).$$
 
For  $m=1,$
the previous model is an analogue on compact Lie groups  of  the quasi-geostrophic Euclidean case that have been considered e.g.  by Caffarelli and Vasseur in  \cite{caf:annals}, where it was considered the case of a differential operator $$D=v\cdot \nabla=\sum_{i=1}^n{v_i}\partial_{x_i},$$ of first order with minimal smoothness condition on $v$.

\begin{remark}
The matrix-valued symbol of $K(t)=a(x,t)\mathcal{L}_{G}^{\frac{m}{2}}+D$ is given by
\begin{equation}
    \sigma_{K(t)}(x,[\xi])=a(x,t)\lambda_{[\xi]}^{\frac{m}{2}} I_{d_\xi\times d_\xi}+\sum_{i=1}a_{i}(x)(X_{i}\xi)(e).
\end{equation} The hypothesis $X_j\in i\mathfrak{g}$ implies that we can choose an element $\xi\in [\xi]$ such that  $\sigma_{X_j}(\xi)=X_j\xi(e)$ can be diagonalised with diagonal entries determining the spectrum of the symbol $\sigma_{X_j}(\xi)$ at any representation space. The smooth functions $a_i$ have been considered real-valued, and if we want that the symbol 
\begin{equation}
    \sigma_{-K(t)}(x,[\xi])=-a(x,t)\lambda_{[\xi]}^{\frac{m}{2}} I_{d_\xi\times d_\xi}-\sum_{i=1}a_{i}(x,t)(X_{i}\xi)(e),
\end{equation} becomes a positive matrix on any representation space is natural to assume that $-a(x,t)\in C^{\infty}(G,\mathbb{R}^+_0)$ is smooth and bounded from below, that is $-a(x,t)\geq 0.$
\end{remark}
In view of Corollaries \ref{elliptic:classes} and  \ref{Main:Th:Sob:2} we have the following well posedness result.
\begin{corollary}\label{Corollary:general:drift} Let $0\leq m\leq 1.$ Let $a(x,t)\in C^{\infty}(G,\mathbb{R}^+_0)$ be a  smooth function such that $a(x,t)\leq 0.$ Let $X_j\in i\mathfrak{g}$ be a family of symmetric vector-fields and for for any $1\leq i\leq n,$ let $a_{i}\in C^{\infty}(G\times [0,T],\mathbb{R})$ be a set of real-valued smooth functions. If, for any $0\leq t\leq T,$ the (drift) first-order differential operator
\begin{equation}\label{general:drift}
    D=\sum_{j=1}^na_j(x,t)X_j\in TG^{\mathbb{C}}\cong (G\times G)\times \mathfrak{g}\otimes i\mathfrak{g},
\end{equation} is such that
\begin{equation}\label{Positivity:condition}
  \sigma_{-K(t)}(x,[\xi])=-a(x,t)\lambda_{[\xi]}^{\frac{m}{2}} I_{d_\xi\times d_\xi}-\sum_{i=1}a_{i}(x,t)(X_{i}\xi)(e)\geq 0,
\end{equation}for any $(x,[\xi])\in G\times\widehat{G}.$
Then, for any  $s\in \mathbb{R},$ $T>0,$    $u_0\in L^2(G)$, and  $f\in L^2([0,T],L^2(G)),$  there exists a unique $v\in C^1([0,T], L^2(G) )) $ solving \eqref{Driff:drift}. Moreover, $v$ satisfies the energy estimate
\begin{equation}\label{eq:energy:fracto}
\Vert v(t) \Vert^2_{L^{2}(G)}\leqslant  \left(C\Vert u_0 \Vert^2_{L^{2}(G)}+C'\Vert f \Vert^2_{L^2([0,T],L^{2}(G))} \right),
\end{equation}for every $0\leqslant t\leqslant T.$  Moreover,  the existence of the solution $v$ in  \eqref{Driff:drift} with $u_0\in H^s(G)$, and  $f\in L^2([0,T],H^s(G)),$ and the energy estimate in \eqref{energyëstimate:ellipt:drift} still are valid if  $D$ is an element of the complex universal enveloping algebra $\mathfrak{U}(\mathfrak{g})\oplus i \mathfrak{U}(\mathfrak{g}) $ and the positivity condition in \eqref{Positivity:condition} holds.
\end{corollary}
 \begin{proof}
 Observe that if $0\leq m\leq 1,$ and {{ $D$ is a left-invariant differential operator of first order as}} in \eqref{general:drift}, under the positivity condition in \eqref{Positivity:condition} the well posedness of  \eqref{Driff:drift} and the validity of the energy estimate in \eqref{eq:energy:fracto} follow from   Corollary \ref{elliptic:classes}. 
 
 In the case where  $D$ is an element of the universal enveloping algebra $\mathfrak{U}(\mathfrak{g})\oplus i \mathfrak{U}(\mathfrak{g}),$ that is, $D$ has constant coefficients, $K(t)$ commutes with the Laplacian $\mathcal{L}_G,$ indeed, $$\mathcal{L}_G\in \textnormal{center}[\mathfrak{U}(\mathfrak{g})\oplus i \mathfrak{U}(\mathfrak{g})],$$ and the well posedness of \eqref{Driff:drift} with $u_0\in H^s(G)$, and  $f\in L^2([0,T],H^s(G)),$  and the validity of the  energy estimate in \eqref{energyëstimate:ellipt:drift} remain valid in view of Corollary \ref{Main:Th:Sob:2}.
 \end{proof}
{{The generalization for all $m>1,$ in the Corollary \ref{general:drift:17}  removes the symmetry condition $X_j\in i\mathfrak{g},$ and the positivity condition in \eqref{Positivity:condition},}} allowing an arbitrary and smooth drift $D$  on $(0,T)$ and taking values on $TG^{\mathbb{C}}.$ This  application of Corollary \ref{Cor:Lstrongly;ellip} can be announced as follows.
\begin{corollary}\label{general:drift:17}Let 
$
    D\equiv D(t)=\sum_{j=1}^na_j(x,t)X_j:[0,T]\rightarrow TG^{\mathbb{C}},
$  be a smooth function and let us consider the (modified) fractional Laplacian with drift $D$
 \begin{equation}
     K(t)=a(x,t)\mathcal{L}_{G}^{\frac{m}{2}}+D,
 \end{equation} where  $m>1.$ Assume that $a(x,t)\in C^{\infty}(G\times [0,T],\mathbb{R}^{+})$ does not have zeros. Then, for any $s\in \mathbb{R}$ the diffusion problem
 \begin{equation}\label{Driff:driftLLaplam1}\begin{cases}\frac{\partial v}{\partial t}=(a(x,t)\mathcal{L}_{G}^{\frac{m}{2}}+D)v+f ,& \text{ }v\in \mathscr{D}'((0,T)\times G)
\\v(0)=u_0 ,& \text{ } \end{cases},
\end{equation} for initial data $u_0\in H^{s}(G)$ and $f\in L^2([0,T],H^{s}(G))$ admits a unique solution $v$ in the space $ C^1([0,T],H^{s}(G))\bigcap C([0,T],H^{s+\frac{m}{2}}(G)).$ Moreover, $v$ satisfies the energy estimate
\begin{equation}
\Vert v(t) \Vert^2_{H^{s}(G)}\leqslant  \left(C\Vert u_0 \Vert^2_{H^{s}(G)}+C'\Vert f \Vert^2_{L^2([0,T],H^{s}(G)))} \right),
\end{equation}for every $0\leqslant t\leqslant T.$ 
\end{corollary}

\subsection{Examples: diffusion problems with  drift for the Laplacian on $\textnormal{SU}(2)$}\label{Ex:SU2section}
Let us consider the compact Lie group of complex unitary $2\times 2$-matrices  
$$ \textnormal{SU}(2)=\{X=[X_{ij}]_{i,j=1}^{2}\in \mathbb{C}^{2\times 2}:X^{*}=X^{-1}\},\,X^*:=\overline{X}^{t}=[\overline{X_{ji}}]_{i,j=1}^{2}. $$
Let us consider the left-invariant first-order  differential operators $$\partial_{+},\partial_{-},\partial_{0}: C^{\infty}(\textnormal{SU}(2))\rightarrow C^{\infty}(\textnormal{SU}(2)),$$ called creation, annihilation, and neutral operators respectively, (see Definition 11.5.10 of \cite{Ruz}) and let us define 
\begin{equation}\label{X3}
    X_{1}=-\frac{i}{2}(\partial_{-}+\partial_{+}),\, X_{2}=\frac{1}{2}(\partial_{-}-\partial_{+}),\, X_{3}=-i\partial_0,
\end{equation}where $X_{3}=[X_1,X_2],$ based on the commutation relations $[\partial_{0},\partial_{+}]=\partial_{+},$ $[\partial_{-},\partial_{0}]=\partial_{-},$ and $[\partial_{+},\partial_{-}]=2\partial_{0},$ the system $X=\{X_1,X_2\}$ satisfies the H\"ormander condition at step $\kappa=2.$

We record that the unitary dual of $\textnormal{SU}(2)$ (see \cite{Ruz}) can be identified as
\begin{equation}
\widehat{\textnormal{SU}}(2)\equiv \{ [\xi^{l}]:2l\in \mathbb{N}, d_{l}:=\dim \xi^{l}=(2l+1)\}\sim ({1}/{2})\mathbb{N}.
\end{equation}
There are explicit formulae for $\xi^{l}$ as
functions of Euler angles in terms of the so-called Legendre-Jacobi polynomials, see \cite{Ruz}.

Consider the basis  $\{X_1,X_2,X_3=[X_1,X_2]\}$ of the Lie agebra $\mathfrak{su}(2).$ The sum of squares
\begin{equation}
    \mathcal{L}_{\textnormal{SU}(2)}=-X_1^2-X_2^2-X_3^2=-\partial_0^2-\frac{1}{2}[\partial_+\partial_{-}+\partial_{-}\partial_{+}],
\end{equation} is the positive Laplacian on $\textnormal{SU}(2).$
Again, by following e.g.  \cite{Ruz},  the spectrum of the positive Laplacian $\mathcal{L}_{\textnormal{SU}(2)}$ can be indexed by the sequence
$$\lambda_\ell:= \ell(\ell+1),\quad \ell\in ({1}/{2})\mathbb{N}.$$
Next, we  illustrate that the positivity condition in Corollary \ref{general:drift}. Consider the problem with drift

\begin{equation}\label{Driff:drift:SU2}\begin{cases}\frac{\partial v}{\partial t}=(a(x,t)\mathcal{L}_{\textnormal{SU}(2)}^{\frac{m}{2}}+ia_3(x,t)X_3)v+f ,& \text{ }v\in \mathscr{D}'((0,T)\times \textnormal{SU}(2)),
\\v(0)=u_0 .& \text{ } \end{cases}
\end{equation} For $m>1,$ if $a(x,t)\in C^\infty(\textnormal{SU}(2)\times [0,T],\mathbb{R}^{+})$ does not have zeros, Corollary \ref{general:drift:17} implies that \eqref{Driff:drift:SU2} is well posed if $u_0\in {H}^s(\textnormal{SU}(2))$ and $f\in L^2([0,T], {H}^s(\textnormal{SU}(2))$ with arbitrary regularity order $s\in \mathbb{R}.$ 

Now, if $0\leq m\leq 1,$ note that the symbol of 
$$-K(t)= a(x,t)\mathcal{L}_{\textnormal{SU}(2)}^{\frac{m}{2}}+ia_3(x,t)X_3  $$
can be computed in diagonal form into a representative $\xi^{\ell}\in [\xi^{\ell}],$ as
\begin{align}\label{positivity:SU2}
    \sigma_{-K(t)}(x,\xi^{\ell})=\textnormal{diag}[-a(x,t)(\ell^2+\ell)^{\frac{m}{2}}-a_{3}(x,t)j_{1}]_{-\ell\leq j_{1},j_2\leq \ell,\,j_i\in (1/2)\mathbb{N}}.
\end{align}Now, note that $\sigma_{-K(t)}(x,\xi^{\ell})\geq 0$ for all $\ell \in (1/2)\mathbb{N},$ if $a_3\equiv 0$ and $0\leq m\leq 1,$ or for $m=1,$ if
\begin{equation}
  a(x,t)\leq 0, \, |a_{3}(x,t)|+a(x,t)\leq0.
\end{equation}
Indeed, the condition in \eqref{positivity:SU2} with $\ell=0=j_1$ implies that $-a(x,t)\geq 0$ for all  $(x,t)\in G\times [0,T].$ If $0\leq m<1,$ note that 
\begin{equation}
    -a(x,t)(\ell^2+\ell)^{\frac{m}{2}}-a_{3}(x,t)j_{1}\geq  -a(x,t)(\ell^2+\ell)^{\frac{m}{2}}-|a_{3}(x,t)|\ell\geq0
\end{equation} for all $\ell\in (1/2)\mathbb{N}$ if and only if $a_3\equiv 0.$ On the other hand, for $m=1,$ 
$$ \forall \ell, -a(x,t)(\ell^2+\ell)^{\frac{m}{2}}-|a_{3}(x,t)|\ell\geq0  $$ if and only if $$\forall \ell,\,-a(x,t)(1+(1/\ell))^{\frac{1}{2}}\geq |a_{3}(x,t)|,$$ from which is required that $|a_3(x,t)|+a(x,t)\leq 0.$ The analysis above and Corollary \ref{general:drift:17} implies the following well posedness result for \eqref{eq:energy:fracto:SU2} showing that the positivity criterion in \eqref{Positivity:condition} is in certain cases verifiable.
\begin{corollary}\label{Corollary:general:drift:su2SU2} Let $0\leq m\leq 1.$ Let $a(x,t)\in C^{\infty}(\textnormal{SU}(2),\mathbb{R}^+_0)$ be a  smooth function such that $a(x,t)\leq 0.$ Let $a_{3}\in C^{\infty}(\textnormal{SU}(2)\times [0,T],\mathbb{R})$ be a real-valued smooth function. If $m=1$, and for any $0\leq t\leq T,$ and all $x\in \textnormal{SU}(2),$
\begin{equation}
   |a_3(x,t)|+a(x,t)\leq 0,
\end{equation} 
then, for any    $u_0\in L^2(\textnormal{SU}(2))$, and  $f\in L^2([0,T],L^2(\textnormal{SU}(2))),$  there exists a unique $v\in C^1([0,T], L^2(\textnormal{SU}(2)) )) $ solving the problem
\begin{equation}\label{Driff:drift:SU2:th}\begin{cases}\frac{\partial v}{\partial t}=(a(x,t)\mathcal{L}_{\textnormal{SU}(2)}^{\frac{m}{2}}+a_3(x,t)iX_3)v+f ,& \text{ }v\in \mathscr{D}'((0,T)\times \textnormal{SU}(2)),
\\v(0)=u_0 ,& \text{ } \end{cases}
\end{equation}where $X_3=-i\partial_0$ is as in \eqref{X3}.
Moreover, $v$ satisfies the energy estimate
\begin{equation}\label{eq:energy:fracto:SU2}
\Vert v(t) \Vert^2_{L^{2}(\textnormal{SU}(2))}\leqslant  \left(C\Vert u_0 \Vert^2_{L^{2}(\textnormal{SU}(2))}+C'\Vert f \Vert^2_{L^2([0,T],L^{2}(\textnormal{SU}(2))} \right),
\end{equation}for every $0\leqslant t\leqslant T.$  Moreover,  the existence of the solution $v$ in  \eqref{Driff:drift:SU2} with $u_0\in H^s(\textnormal{SU}(2))$, and  $f\in L^2([0,T],H^s(\textnormal{SU}(2))),$ and the energy estimate in \eqref{eq:energy:fracto:SU2} still are valid if $a(x,t)\equiv c_1$ and $a_{3}(x,t)\equiv a_3$ are (non-trivial) constant functions (at the same time) such that $|a_3|+c_1\leq 0$ with $c_1\leq 0.$ On the other hand, the conclusions above still remain valid when $0\leq m<1$ if $a_3\equiv 0.$
\end{corollary}

\bibliographystyle{amsplain}

\end{document}